\documentclass[reqno,10pt]{article}
\usepackage{amssymb,amsmath}
\usepackage{amsthm,eucal}
\usepackage{amsfonts}
\usepackage{mathrsfs,enumerate}
\usepackage[usenames]{color}
\usepackage{a4wide}
\usepackage[normalem]{ulem}

\definecolor{refkeybis}{gray}{.65}% per avere le labels stampate chiare
\definecolor{labelkeybis}{gray}{.65}% basta modificare .50:= intensita grigio
{\makeatletter
\def\SK@refcolor{\color{refkeybis}}%
\def\SK@labelcolor{\color{labelkeybis}}}

\numberwithin{equation}{section} % per cambiare la numerazione
\newtheorem{theorem}{Theorem}[section] % per cambiare la numerazione
\newtheorem{lemma}[theorem]{Lemma}
\newtheorem{corollary}[theorem]{Corollary}
\newtheorem{proposition}[theorem]{Proposition}

\theoremstyle{definition}
\newtheorem{remark}[theorem]{Remark}
\newtheorem{example}[theorem]{Example}

\newtheorem{definition}[theorem]{Definition}

\newcommand{\beq}{\begin{equation}}
\newcommand{\eeq}{\end{equation}}
\newcommand{\R}{\mathbb{R}}

\newcommand{\N}{\mathbb{N}}

\newcommand{\DD}{\mathscr{D}}

\newcommand{\ii}{{\mbox{\boldmath$i$}}}

\newcommand{\rr}{{\mbox{\boldmath$r$}}}

\renewcommand{\ss}{{\mbox{\boldmath$s$}}}

\newcommand{\vv}{{\mbox{\boldmath$v$}}}

\newcommand{\tauV}{{\kern-3pt\tau}}

\newcommand{\srr}{{\mbox{\scriptsize\boldmath$r$}}}

\newcommand{\rrho}{{\mbox{\boldmath$\rho$}}}

\newcommand{\xxi}{{\mbox{\boldmath$ \xi$}}}

\renewcommand{\restriction}[1]{\lower3pt\hbox{$|_{#1}$}}

\newcommand{\de}{\partial} %derivata parziale
\newcommand{\supp}{\mathop{\rm supp}\nolimits} % supporto

\newcommand{\Leb}[1]{{\mathscr L}^{#1}} % Misura di Lebesgue
\newcommand{\weakto}{\rightharpoonup}
\newcommand{\eps}{{\varepsilon}}
\renewcommand{\to}{\rightarrow}
\newcommand{\PlusMeasuresTwo}[1]{\mathscr M_{2}(#1)} % misure non negative
\newcommand{\PlusMeasures}[1]{\mathscr M_+(#1)} % misure non negative
\newcommand{\PlusMeasuresComp}[1]{\mathscr M_+^{\rm comp}(#1)} % misure non negative a supporto compatto
\newcommand{\CPlusMeasures}[1]{\mathscr M_+^{\rm c}(#1)} % misure non negative
\newcommand{\CPlusMeasuresTwo}[1]{\mathscr M_2^{\rm c}(#1)} % misure non negative
\newcommand{\rme}{\mathrm e}
\newcommand{\rmc}{\mathrm c}
\newcommand{\sfE}{\mathsf E}
\newcommand{\sfD}{\mathsf D}

\newcommand{\stD}{\tilde\sfD}

\newcommand{\sfR}{\mathsf R}
\newcommand{\sfG}{\mathsf G}

\newcommand{\cI}{\mathcal I}

\newcommand{\cE}{\mathcal E}
\newcommand{\cF}{\mathcal F}
\newcommand{\cV}{\mathcal V}
\newcommand{\cG}{\mathcal G}

\newcommand{\Mom}[2]{{\mathfrak m}_{#1}(#2)}

\renewcommand{\d}{{\mathrm d}}
\newcommand{\dx}{\d x}

\newcommand{\Rd}{{\R^d}}
\newcommand{\loc}{{\mathrm{loc}}}

\newcommand{\sfS}{\mathsf{S}}
\newcommand{\sfX}{\mathsf{X}}

\newcommand{\sfc}{\mathsf{c}}

\newcommand{\restr}[1]{\lower3pt\hbox{$|_{#1}$}}
\newcommand{\mass}{{\mathfrak m}}

\newcommand{\betainfty}{{\beta^{\infty}}}

\newcommand{\Einfty}{{E^{\infty}}}

\newcommand{\nchi}{{\raise.3ex\hbox{$\chi$}}}

\newcommand{\Dom}[1]{\mathsf{D}(#1)}
\newcommand{\Domp}[1]{\mathsf{D}_+(#1)}
\newcommand{\Pos}[1]{\Omega_+(#1)}
\newcommand{\Conn}[1]{\mathscr I(#1)}
\newcommand{\CDom}[1]{\mathsf{J}(#1)}
\newcommand{\dV}{\frak d}

\newcommand{\urho}{{\rho}}
\makeindex
\begin{document}
%\screencolor

\title{\bf Measure valued solutions of sub-linear diffusion equations
  with a drift term
% \thanks{This work has been partially supported by the IMATI (CNR Pavia, Italy)}
}

\author{S. Fornaro\thanks{Dipartimento di Matematica ``F. Casorati'', Universit\`a
    degli Studi di Pavia, Italy. E-mail: \textsf{simona.fornaro@unipv.it}, \textsf{stefano.lisini@unipv.it},  \textsf{giuseppe.savare@unipv.it},  \textsf{giuseppe.toscani@unipv.it}.}
    \qquad S. Lisini\footnotemark[1]
    \qquad G. Savar\'e\footnotemark[1]
\qquad G. Toscani\footnotemark[1]}

\date{\today}

\maketitle

\begin{abstract}
In this paper we study nonnegative, measure valued solutions of the
initial value problem for one-dimensional drift-diffusion equations
when the nonlinear diffusion
is governed by an increasing $C^1$ function $\beta$ with
$\lim_{r\to +\infty} \beta(r)<+\infty$.
By using tools of optimal transport, we will show that 
this kind of problems is well posed in the class of
nonnegative Borel measures with finite mass $\frak m$ and finite
quadratic momentum and it is the gradient flow of
a suitable entropy functional with respect to the so called $L^2$-Wasserstein distance.

Due to the degeneracy of diffusion for large densities,
  concentration of masses can occur, whose support
  is transported by the drift.
We shall show that the large-time behavior of solutions depends
on a critical mass $\mass_{\rm c}$, which can be explicitely characterized in
terms of $\beta$ and of the drift term.
If the initial mass is less then $\mass_{\rm c}$,
the entropy has a unique minimizer which is absolutely continuous with
respect to the Lebesgue measure.

Conversely, when the total mass $\mass$ of the solutions is greater than the critical one,
the steady state has a singular part in which
the exceeding mass $\mass- \mass_{\rm c}$ is accumulated.
\end{abstract}

\noindent {\em Keywords}: sublinear diffusion, concentration
phenomena, propagation of singularities, gradient flows,
nonlinear diffusion equations, Wasserstein
distance, measure valued solutions.

\section{Introduction}
In this paper we study nonnegative, measure-valued solutions of
the Cauchy problem for a one-dimensional drift-diffusion
equation
\begin{equation}\label{eq:rho}
    \de_t\urho-\de_x\big(\de_x(\beta(\urho))+V'\urho\big)=0 \quad {\rm in }\ (0,+\infty)\times  \R,
    \qquad
    \urho(0,\cdot)=\urho_0 \quad \text{in } \R.
    \tag{1.DDE}
\end{equation}
Here we assume that
\begin{equation}\label{hp:beta}
  \beta\in C^1([0,+\infty)) \text{ is increasing,}\quad
  \beta(0)=0,\quad
  \betainfty:=\lim_{r\to +\infty} \beta(r)<+\infty,
  \tag{1.$\beta$}
\end{equation}
and $V:\R\to \R$ is a $C^2$ driving potential,
satisfying the conditions
% \begin{equation}\label{hp:V}
%   V\in C^1(\R),\qquad V \text{ is convex } \qquad \min_{x\in\R} V(x)=V(0)=0,
% \end{equation}
% and the quadratic growth at $\infty$, i.e., there exist two positive constants $a,b$ such that
\begin{equation}   \label{crescita-quadratica}
 V''(x)\ge \lambda\quad  \quad {\rm for\  all}\ x\in \R;\qquad
 {\liminf_{|x|\to+\infty}\frac{V(x)}{|x|^2}\ge0.}
% \footnote{Serve per la $\Gamma$-convergenza}
 % |V'(x)|\le \sfa+\sfb\,|x|,
 \tag{1.$V$}
\end{equation}
%for suitable nonnegative constants $\sfa,\sfb\ge0$.
We will look for solutions $t\mapsto \rho_t$
in the space $\PlusMeasuresTwo{\R,\frak m}$ of nonnegative Borel measures with finite mass $\frak m=\urho(\R)$ and finite
quadratic momentum
\begin{equation}
  \label{eq:7}
 \Mom2\urho:=\int_\R |x|^2\,\d\urho(x)<+\infty.
\end{equation}

%%%%%%%%%%%%%%%%%%%%%%%%%%%%%%%%%%%%%%%%%%%%%%%%%%%%%%%%%%%%%%%%%%%%%%%%%%%%%%%%%%

Conditions \eqref{hp:beta} describe the physical situation in which
the diffusion operator is very weak and {possibly} unable to smooth
out the solution if initially point masses are present.

  This fact is reflected by the natural entropy functional $\cF$ which
  generates equations like \eqref{eq:rho} as gradient flow in
  $\PlusMeasuresTwo{\R,\frak m}$ and in particular
  decays along the solutions of \eqref{eq:rho},
  % \footnote{La condizione $rE''=\beta'$ e $E(0)=0$ non basta per
  %   determinare univocamente l'entropia}
  \begin{equation}
    \label{eq:29}
    \cF(\rho):=\cE(\rho)+\cV(\rho),\quad
    \cE(\rho):=\int_\R E(u(x))\,\d x\quad\text{if }\rho=u\Leb
    1+\rho^\perp,\quad
    \cV(\rho)=\int_\R V(x)\,\d\rho(x),
  \end{equation}

where
{the convex energy density function $E:[0,+\infty)\to\R$ is defined as}
\begin{equation}
  \label{eq:26}
  E(r):=-\beta(r)-r\int_r^{+\infty} \frac{\beta'(s)}s\,\d s\quad
  \text{so that}\quad \beta'(r) = rE''(r),\quad
  E(0)=0,
  \tag{1.$E$}
\end{equation}
and satisfies
%\GGr{belongs to the set of continuous functions such that the asymptote function}
\begin{equation}
 \lim_{r\to+\infty}E(r)=-\beta_\infty\quad\text{and therefore}\quad
 \lim_{r \to+\infty} \frac{E(r)}r=0,\label{eq:32}
\end{equation}
 so that the (lower semicontinuous) integral functional $\cE$ defined
 by \eqref{eq:29}
 {depends only on}
 %\GGr{can see only}
 the regular part of a Borel measure (see for instance
\cite{DT84}).

  Even worst, the energy density $E$ does not satisfy the
  regularizing condition \cite[Thm. 10.4.8]{ags}
  $\lim_{r\to+\infty}E(r)=-\infty$, which prevents a singular part
  for measures with finite energy dissipation along \eqref{eq:rho},
  thus in particular for any solution $\rho_t$ at positive time $t>0$.

\paragraph{Sub-linear diffusions and Bose-Einstein distribution.}
In order to better clarify the physical meaning of condition
\eqref{hp:beta}, let us briefly describe a situation in $\Rd$
in which the steady state of the
drift-diffusion equation is explicitly computable.
To this aim, for  $x \in \R^d$, $d \ge 1$, let us fix $V(x) = |x|^2/2$, while, for a fixed
constant $\alpha >0$, the diffusion function $\beta(r)$ is defined by
 \beq\label{alpha}
 \beta(0)=0,\qquad \beta'(r) = \frac 1{1+ r^\alpha}.
 \eeq
Then, since in this case the drift-diffusion equation
 \[
\de_t\urho-\nabla_x \cdot\left( \nabla_x\beta (\urho) + x \urho\right)=0, \qquad x \in \R^d
 \]
can be rewritten as
 \beq\label{gen}
\de_t\urho-\nabla_x \cdot \left( \rho \nabla_x\left(\frac 1\alpha \ln \frac{\rho^\alpha}{1+ \rho^\alpha} + \frac{|x|^2}2\right)\right)=0,
 \eeq
the steady states of \eqref{gen} are given by

 \beq\label{steady-gen}
 \rho_\infty (x) =  \left[ e^{\alpha |x|^2/2 +\eta} - 1 \right]^{-1/\alpha}, \qquad \eta \ge 0.
 \eeq
The (nonnegative) constant $\eta$ in \eqref{steady-gen} identifies the mass of the steady solution
 \[
\mass_\eta = \int_{\R^d} \left[ e^{\alpha |x|^2/2 +\eta} - 1
\right]^{-1/\alpha} \, \d x.
 \]
Since the mass $\mass_\eta$ is decreasing as soon as $\eta$ increases, the maximum value of $\mass_\eta$ is attained at $\eta =0$.
Note that, if $B_d$ denotes the measure of the unit sphere in $\R^d$,  the value
  \[
\mass_0 = \int_{\R^d} \left[ e^{\alpha |x|^2/2} - 1 \right]^{-1/\alpha} \, \d x = B_d \int_0^{+\infty} r^{d-1} \left[ e^{\alpha r^2/2} - 1
\right]^{-1/\alpha} \, \d r
 \]
is bounded as soon as $\alpha > 2/d$. Whenever the constant $\alpha$ is chosen in this range, the value
 \beq\label{crit}
\mass_{\rm c} = \mass_0 = B_d \int_0^{+\infty} r^{d-1} \left[ e^{\alpha r^2/2} - 1 \right]^{-1/\alpha} \, \d r < +\infty
 \eeq
defines the so-called \emph{critical mass} of the problem, namely the maximal mass that can be achieved by a regular steady state. It is
interesting to remark that, in view of the lower bound on $\alpha$ which implies the existence of a critical mass, since in dimension one $\alpha
>2$, the function $\beta$ in \eqref{alpha} satisfies conditions \eqref{hp:beta}, in particular
 \[
\lim_{r\to +\infty} \beta(r)<+\infty.
 \]
This condition clearly can fail in higher dimensions.

 The most relevant physical example of such type of steady states is furnished by the three-dimensional Bose-Einstein distribution
\cite{CC70}
 \beq\label{BEsteady}
 u_\infty (x) =  \left[ e^{|x|^2/2 +\eta} - 1 \right]^{-1}
 \eeq
that is the steady state of equation \eqref{gen} corresponding to $\alpha = 1$. In this case the function $\beta$ is explicitly computable to give
$\beta(\urho)= \ln (1+\urho)$. Since $\alpha =1$, if the dimension $d \ge 3$, the Bose-Einstein distribution exhibits a {critical mass}. We remark
that in this case the energy functional $E(u)$ is the Bose-Einstein entropy
 \[
E(u) = u\ln u - (1+u) \ln(1+u).
 \]
One of the fundamental problems related to evolution equations that
relax towards a stationary state characterized by the existence of a
critical mass, is to show how, starting from an initial distribution
with a supercritical mass $\mass> \mass_{\rm c}$, the solution eventually develops
a singular part (the condensate), and, as soon as the singular part
is present, to be able to follow its evolution. We remark that in
general the condensation phenomenon is heavily dependent of the
dimension of the physical space. In dimension $d \le 2$, in fact,
the maximal mass $\mass_0$ of the Bose-Einstein distribution
\eqref{BEsteady} is unbounded, and the eventual formation of a
condensate is lost.

In order to simplify the mathematical difficulties, while maintaining the physical picture in which the steady state has a critical mass, in
\cite{BGT09} the one-dimensional case corresponding to a steady state of the form \eqref{steady-gen}, with $\alpha >2$ has been considered.  Note
that the analysis of \cite{BGT09} refers to a linear diffusion with a super-linear drift
 \[
\de_t \rho = \de_x\left( \de_x \rho  + x \rho(1+ \rho^\alpha) \right),
 \]
that is reminiscent of the Kaniadakis-Quarati model of Bose-Einstein particles \cite{KQ94}
 \begin{equation}
\label{FPB} \de_t \rho = \nabla\cdot \left( \nabla \rho + x \rho(1+ \rho)\right).
 \end{equation}

   \paragraph{A measure-theoretic setting for diffusion equations.}
   \
In the present paper we deal with an almost  complete description of the time-evolution of the solution of problem \eqref{eq:rho} with a Borel
measure as initial datum. While the mathematical study of drift-diffusion kinetic equations with the Bose-Einstein density as steady state has
been considered before (cfr. \cite{EHV98, EMV03} and the references therein), to our knowledge, drift-diffusion equations of type \eqref{eq:rho} at
present have never been studied systematically.

  Motivated by the previous remarks and by the degeneracy of the entropy functional $\cF$ introduced in
  \eqref{eq:29}, whose minimizers could exhibit concentration effect,
  we address the study of \eqref{eq:rho} by the measure-theoretic
  point of view recently developed in the framework of optimal
  transport \cite{ags}. This approach, started by the pioneering papers
  of \textsc{Jordan-Kinderlehrer-Otto} \cite{Jordan-Kinderlehrer-Otto98} and \textsc{Otto} \cite{Otto01},
  %\footnote{(aggiungere le citazioni con un file .bib)},
  provides a sufficiently general setting for measure-valued
  solutions to \eqref{eq:rho}.

  $\PlusMeasuresTwo{\R,\frak m}$ endowed
  with the so called $L^2$-Wasserstein distance is the natural ambient
  space
  for carrying on our analysis.
  A first important fact is that
  the entropy functional $\cF$ \eqref{eq:29} turns out to be
  displacement $\lambda$-convex, a crucial property
  which holds only in the one-dimensional case, since the possibility
  of entropies satisfying \eqref{eq:32} is prevented by
  \textsc{McCann}'s condition
  \cite{McCann97}
  in higher dimensions.

  Moreover, we are able to extend the results of \cite{ags} (which
  for sublinear entropies covers the case when $\lim_{r\to+\infty}E(r)=-\infty$)
  providing an explicit characterization of the
  dissipation of $\cF$, which is strictly related to
  the ``Wasserstein differential'' of $\cF$.
  As a crucial byproduct of this analysis, we will find
  the right condition that measure-valued solutions have to satisfy
  in order to enjoy nice uniqueness and stability results. It
  is worth mentioning here that the distributional formulation of \eqref{eq:rho}
  does not provide enough information to characterize the solutions,
  when a concentration on a Lebesgue negligible set occurs.

  Applying the general theory
  of gradient flows of displacement $\lambda$-convex functionals in
  Wasserstein spaces, we can thus obtain a precise
  characterization
  of measure valued solutions to \eqref{eq:rho} and we can prove their
  existence, uniqueness, and stability.

  Further justifications showing that the notion of Wasserstein
  solutions is well adapted to \eqref{eq:rho} come from natural
  regularization/approximation results: we will show that
  our solutions are both the limit of the simplest vanishing viscosity
  approximation
  of \eqref{eq:rho} and of smooth solutions generated by
  regularization of the initial data.

  We complete our analysis by studying the propagation of the
  singularities, the structure of minimizers of $\cF$ and of
  stationary solutions,
  and the asymptotic behavior of the solutions, showing general
  convergence results to the minimizer of $\cF$.
  % Further remarks about the role of the diffusion function $\beta$
  % and the role of the potential $V$ for the existence of a
  % finite critical mass are given after Theorem \ref{thm:main4}
  % and Examples \ref{ex:1}, \ref{ex:2}.

  \paragraph{Plan of the paper.}
  In the next section we will make precise our definition
  of measure-valued solutions to \eqref{eq:rho} (\S
  \ref{subsec:def})
  and we
  will present our main results concerning existence, uniqueness,
  stability, and approximation of Wasserstein
  solutions (\S \ref{subsec:main}).
  The equation governing the propagation of their singularities is
  considered
  in \S \ref{subsec:propagation}.
  \S \ref{subsec:minimizers} is devoted to a precise characterization of
  minimizers
  of $\cF$ and of the critical mass; steady states are studied in \S
  \ref{subsec:stationary} and \S \ref{subsec:asymptotic}
  collects some results concerning the
  asymptotic behaviour of Wasserstein solutions.

  Section \ref{sec:Wass} briefly recalls some definitions and tools
  of (one-dimensional) optimal transport, Wasserstein distance, and
  the related (sub)differentiability properties of displacement
  $\lambda$-convex functionals. Theorems \ref{thm:lsc-dissipation}
  and \ref{th:charsubdiff} lie at the core
  of our further developments. A last paragraph
  devoted to a simple regularization of $\cF$ by
  $\Gamma$-convergence concludes the section.

  The last section contains the proofs of all our
  main results: the connection with
  the general theory is discussed in \S \ref{sec:gf} and
  \S \ref{sec:estimates} is devoted to the propagation
  of the singularities;
  the study of the minimizers of $\cF$ and of the
  related asymptotic behavior
  of the solutions to \eqref{eq:rho} is
  performed in the last part.

%%%%%%%%%%%%%%%%%%%%%%%%%%%%%%%%%%%%%%%%%%%%%%%%%%%%%%%%%%%%%%%%%%%%%%%%%%%%%%%%%%%%

\section{Definitions and main results}

In this section we collect the main
definitions and results we shall prove in the rest of the paper.

\subsection{Wasserstein solutions to \eqref{eq:rho}}
\label{subsec:def}
\paragraph{The case of bounded initial densities and Lipschitz drifts.}
When (the Lebesgue density of) $\urho_0\in L^\infty(\R)$ and the potential $V$ is such that
\begin{equation}
  \label{eq:8}
  V''(x)\le \sfc\quad\text{for every }x\in \R,
\end{equation}
it is not difficult {(see \cite{Vazquez07} and next Corollary \ref{cor:linfty_bound})}
to show that {a smooth} solution $\urho_t$
{of \eqref{eq:rho}}
satisfies the \emph{a priori} estimate
\begin{equation}
  \label{eq:9}
  \sup_{t\in [0,T]}\|\urho_t\|_{L^\infty(\R)}\le R_T:=\|\urho_0\|_{L^\infty(\R)}{\mathrm e}^{\sfc\, T}\quad\text{for every }T>0,
\end{equation}
so that it is uniformly bounded in every bounded time interval $[0,T]$. We can infer from \eqref{eq:9} that the behavior of $\beta(r)$ as
$r\uparrow+\infty$ does not play any role,  and a solution in $[0,T]$ could be easily obtained by solving \eqref{eq:rho} with respect to a
nonlinearity $\tilde\beta$ defined for instance by
\begin{equation}
  \label{eq:10}
  \tilde\beta(r):=
  \begin{cases}
    \beta(r)&\text{if }r\le 2 R_T,\\
    \beta(2R_T)+\beta'(2R_T)(r-2 R_T)&\text{if }r>2R_T.
  \end{cases}
\end{equation}
Denoting by $\sfS_t(\urho_0)$ the solution $\urho_t$ generated by a bounded initial datum $\urho_0$, it is possible to check that $\sfS_t$ satisfies
the $L^1$ contraction property
\begin{equation}
  \label{eq:11}
  \big\|\sfS_t(\urho)-\sfS_t(\eta)\big\|_{L^1(\R)}\le \|\urho-\eta\|_{L^1(\R)}\quad\text{for every }\urho,\eta\in L^1(\R)\cap L^\infty(\R).
\end{equation}
Consequently $\sfS_t$ can be extended in a canonical way to a contraction semigroup in the cone $L^1_+(\R)$ of nonnegative integrable densities.

\paragraph{Measure-valued solutions.}
In case the Lebesgue density of $\urho_0$ is not bounded
{or $V$ does
not satisfy \eqref{eq:8}}, the presence of a singular part in the
solution $\urho$ of \eqref{eq:rho} has to be taken into account, since
%\GGr{In fact, as briefly discussed in the second paragraph,}
the boundedness
of $\beta$ is responsible of the (possible) presence of a critical
mass. {We shall see an example of a solution $\rho_t$ exhibiting a
  singular part for every $t\ge0$ in the next Remark \ref{rem:singularex}.}

{In the following we will denote by $\PlusMeasures\R$ (resp.\
  $\PlusMeasures{\R,\mass}$) the space of nonnegative Borel
  measures in $\R$ with finite mass (resp.\ with prescribed mass
  $\mass>0$)
  and by $\PlusMeasuresTwo\R$ (resp.\ $\PlusMeasuresTwo{\R,\mass}$) the
  collection of measures in $\PlusMeasures\R$ (resp.\ in
  $\PlusMeasures{\R,\mass}$) with finite quadratic
  momentum.}
In order to enucleate a precise notion of measure-valued solution, for every $\urho\in \PlusMeasures\R$ we
consider the classical Lebesgue decomposition
\begin{equation}
  \label{eq:1}
  \urho=\urho^a+\urho^\perp,\quad
  \urho^a=u\,\Leb1,
\end{equation}
where $u\in L^1(\R)$ is the Lebesgue density of the absolutely continuous part $\urho^a$ of $\urho$ and $\urho^\perp$ is the  singular part of
$\urho$, concentrated on a set of Lebesgue measure $0$.

It is then natural to substitute the term $\beta(\urho)$ in \eqref{eq:rho} by $\beta(u)$ and then interpret \eqref{eq:rho} in the sense of
distributions. If we want %\GGr{both to keep all the information and}
to obtain a good notion of solution,
we should add some further requirements to the
density $u$. The first one is of qualitative type,
and relies in considering $u$ as a continuous function on $\R$ with values in the  extended set
$[0,+\infty]$, endowed with the usual topology.
\begin{definition}[Measures with continuous densities]
  \label{def:right_measures}
  We say that a measure $\urho=\urho^a+\urho^\perp\in \PlusMeasures\R$ has
  a generalized continuous density $u\in C^0(\R;[0,+\infty])$ with
  proper
  domain $\Dom u:=\big\{x\in \R:u(x)<+\infty\big\}$ if
  \begin{align}
    \label{eq:2}
    \urho^\perp\big(\Dom u\big)=0,\quad
    \Leb 1\big(\R\setminus \Dom u\big)=0,\quad\text{and}\quad
    \urho^a=u\,\Leb 1\restr{\Dom u}.
  \end{align}
  {We set $\Domp u:=\big\{x\in \sfD(u):u(x)>0\big\}$.}
  We denote by $\CPlusMeasures\R$ the collection of all measures with
  generalized continuous density and we set
  $\CPlusMeasuresTwo\R:=\CPlusMeasures\R\cap \PlusMeasuresTwo\R$,
  $\CPlusMeasuresTwo{\R,\mass}:=\CPlusMeasures\R\cap \PlusMeasuresTwo{\R,\mass}$.
\end{definition}
Notice that $\Dom u$ is a \emph{dense open} subset of $\R$, $\urho^\perp=\urho\restr{\R\setminus \Dom u}$, and
\begin{equation}
  \label{eq:4}
  \lim_{x\to \bar x}u(x)=+\infty\quad\text{for every }\bar
  x\in \partial \Dom u=\R\setminus
  \Dom u.
\end{equation}
In particular, $\CPlusMeasures\R$ does not contain any purely singular measure: if $\urho^a=0$ then also $\urho^\perp$ vanishes.

If $\urho\in \CPlusMeasures\R$ then we will always identify its Lebesgue density $\d\urho/\d\Leb1$ with the (unique) continuous precise
representative $u\in C^0(\R;[0,+\infty])$  given by Definition \ref{def:right_measures}. By \eqref{hp:beta} we can consider $\beta$ as a
continuous function defined on the extended set $[0,+\infty]$ and therefore the composition $\beta\circ u$ is a well defined real continuous
function on $\R$.

The second requirement is a quantitative estimate concerning
%\GGr{the boundedness of}
the ``generalized Fisher'' dissipation functional.

\begin{definition}[Generalized Fisher dissipation]
  If $\urho$ belongs to $\CPlusMeasures\R$ with continuous density $u$
  we set
  \begin{equation}
    \label{eq:5}
    {\cI}(\urho):=\int_{{\Domp u}} \Big|\frac{\partial_x \beta(u)}{u}+V'\Big|^2 u\,\d
    x+\int_{\R} |V'|^2\,\d\urho^\perp\quad \text{if}\quad\beta\circ
    u\in W^{1,1}_{\rm loc}(\R).
  \end{equation}
  When $\beta\circ u\not\in W^{1,1}_{\rm loc}(\R)$ or $\urho\not\in
  \CPlusMeasures\R$, we simply set $\cI(\urho):=+\infty.$
\end{definition}
It turns out that $\cI$ is a lower semicontinuous functional with respect to weak
  convergence of measures in $\PlusMeasures\R$
  (see Theorem \ref{thm:lsc-dissipation2}).
\begin{definition}[Wasserstein solutions to \eqref{eq:rho}]
  \label{def:measuresolution}
  We say that $\urho\in C^0([0,+\infty);\PlusMeasuresTwo{\R,\mass})$
  is a Wasserstein solution of problem \eqref{eq:rho} if,
  denoting by $\urho_t$ the measure $\urho$ at the time $t$,
  \begin{subequations}
    \begin{equation}
    \label{eq:3}
    \urho_t\in \CPlusMeasures\R \text{ for $\Leb 1$-a.e.\ $t>0$,}
  \end{equation}
  \begin{equation}\label{eq:3bis}
    \int_{T_0}^{T_1} \cI(\urho_t)\,\d t<+\infty\quad\text{for every } 0< T_0<T_1<+\infty,
  \end{equation}
  and
  \begin{equation}\label{eq-debole}
    \int_0^{+\infty} \int_\R\Big(-\de_t\varphi+ \de_x\varphi V'\Big)
    \,\d\urho_t\,  \d t+
    \int_0^{+\infty}\int_\R\de_x\varphi\, \de_x\beta(u_t)\,\d
    x\,\d t
    % +\int_0^{+\infty}\int_\R\de_x\phi V'(x)\,d\mu_t\,dt
    =0 \quad \forall \varphi\in C_{\rm c}^\infty((0,+\infty)\times \R),
  \end{equation}
  \end{subequations}
  where $u_t$ is the generalized continuous density of $\urho_t$ for
  $\Leb1$-a.e. $t>0$.
\end{definition}

\begin{remark}[Convergence in $\PlusMeasuresTwo{\R,\frak m}$]
  {$\PlusMeasuresTwo{\R,\mass}$ is a complete metric space endowed with}
    %\GGr{We denote by $W_2(\urho,\eta)$}
  the so called $L^2$-Wasserstein
    distance $W_2(\cdot,\cdot)$.
    %\GGr{between two measures $\urho,\eta\in
    % \PlusMeasuresTwo{\R,\frak m}$ with the same mass $\frak m$.}
  More details on this distance will be given in the next section;
  let us just recall that a sequence $\urho_n$
  converges to $\urho$ in $\PlusMeasuresTwo{\R,\frak m}$ as $n\uparrow+\infty$ iff
  \begin{equation}
    \label{eq:14}
    \lim_{n\uparrow+\infty}\int_\R\varphi(x)\,\d\urho_n(x)=
    \int_\R\varphi(x)\,\d\urho(x)\quad\text{for every }\varphi\in
    C^0(\R)\text{ with }
    \sup_x\frac{|\varphi(x)|}{1+x^2}<+\infty.
  \end{equation}
\end{remark}

\begin{remark}[The role of the generalized continuous density]
  By neglecting condition \eqref{eq:3} one can easily  construct
  evolutions of purely
  singular measures which solve \eqref{eq-debole} and are not
  influenced at all by the diffusion term. We
  %\GGr{One way to obtain this is to}
  take a finite number of $C^1$ curves $x_j:[0,+\infty)\to
  \R$, $i=1,\cdots,N$,
  which solve the differential equation $\dot x_j(t)=-V'(x_j(t))$ in
  $[0,+\infty)$, and we set
  \begin{equation}
    \label{eq:6}
    \urho_t:=\sum_{j=1}^N \alpha_j \delta_{x_j(t)}, \qquad \alpha_j\ge0.
  \end{equation}
  In this case $\urho^a_t\equiv0$ for every $t\ge0$, which implies
  $\beta(u_t)\equiv0$ and \eqref{eq-debole} contains just the pure
  transport contribution given by the first integral.
  On the other hand, by taking a smooth approximating family
  $\urho^\eps \to \urho_0$ in $\PlusMeasuresTwo{\R,\frak m}$, we can see that \eqref{eq:6} is not the
  limit of the corresponding solution $\urho^\eps_t$ as $\eps\downarrow0$ (see Theorem \ref{thm:main1}).
\end{remark}

\paragraph{Energy functional and Fisher dissipation.}
In order to understand both the role of the generalized Fisher
dissipation and the consequences of \eqref{eq:3bis},
let us recall the definition \eqref{eq:26} of the so-called
internal energy density $E:[0,+\infty)\to \R$ by the relation
\begin{equation}\label{def:E}
  E(r):=-\betainfty+
  \int_r^{+\infty} \big(1-\frac{r}s\big)\beta'(s)\,\d  s
  =
  -\beta(r)-
  r\int_r^{+\infty}\frac{\beta'(s)}s\,\d  s.
\end{equation}
It is simple to check that $E$ is a convex nonpositive function satisfying
\begin{equation}\label{eq:12}
  E\in C^2(0,+\infty),\quad
  E(0)=0, \quad
  \lim_{r\to 0^+}\frac{E(r)}{r\log r} = \beta'(0) \in [0,+\infty), \quad
  \Einfty=\lim_{r\uparrow+\infty}E(r)=-\betainfty,
\end{equation}
and
\begin{equation}\label{eq:121}
  \beta'(r)=rE''(r), \quad
  E'(r)=-\int_r^{+\infty}\frac{\beta'(s)}s\,\d  s < 0, \qquad \forall r \in (0,+\infty).
\end{equation}
We associate the integral functional
\begin{equation}
  \label{eq:54}
  \cE(\urho):=\int_\R E(u(x))\,\d x \quad   \text{whenever}\quad
  \urho=u\Leb1 +\urho^\perp\in \PlusMeasures\R,
\end{equation}
to the energy density $E$, the potential energy
\begin{equation}
  \label{eq:83}
  \cV(\rho):=\int_\R V(x)\,\d\urho(x)
\end{equation}
to the potential $V$,
and the energy functional $\cF:\PlusMeasures{\R}\to (-\infty,+\infty]$
\begin{equation}\label{def:EE}
\cF(\urho):=\cE(\urho)+\cV(\rho).
\end{equation}
Formal computations show that $\cF$ and $\cI$
satisfy the energy dissipation identity along solutions to \eqref{eq:rho}
\begin{equation}
  \label{eq:13}
  \cF(\urho_{t_1})+\int_{t_0}^{t_1} \cI(\urho_t)\,\d t=\cF(\urho_{t_0}) \qquad 0\leq t_0<t_1<+\infty.
\end{equation}

\subsection{Existence, stability, and approximation results.}
\label{subsec:main}
% The result of existence and uniqueness of solutions of problem \eqref{eq:rho} is stated in the following

  Recall that $\lambda\in \R$ is a lower bound for the second
  derivative of $V$, see \eqref{crescita-quadratica}. Let us set
  \begin{equation}
    \label{eq:27}
    \sfE_\lambda(t):=\int_0^t \rme^{\lambda s}\,\d s=
    \begin{cases}
      \frac{\rme^{\lambda t}-1}\lambda&\text{if }\lambda\neq 0,\\
      t&\text{if }\lambda=0.
    \end{cases}
  \end{equation}

\begin{theorem}[Existence, uniqueness, stability, and comparison]
  \label{thm:main1}
  For every $\urho_0\in \PlusMeasuresTwo{\R,\mass}$ there exists a
  unique Wasserstein solution $\urho_t$ to \eqref{eq:rho} according to Definition
  \ref{def:measuresolution}. This solution satisfies the regularization
  estimate
  \begin{equation}
    \label{eq:18}
    \cF(\urho_t)+\frac{\sfE_\lambda(t)}2\cI(\urho_t)\le \frak
    m V(0)+\frac 1{2\sfE_\lambda(t)}\Mom 2{\urho_0}\quad\text{for every }t>0,
  \end{equation}
  the energy dissipation identity \eqref{eq:13}, and the dissipation inequality
  \begin{equation}
    \label{eq:19}
    \cI(\urho_t)\le \cI(\urho_{t_0}) \rme^{-2\lambda (t-t_0)}, \qquad \forall\, t\geq t_0\ge0.
  \end{equation}
  The map $\sfS_t:\PlusMeasuresTwo{\R,\mass}\to\PlusMeasuresTwo{\R,\mass}$
  defined by $\sfS_t(\urho_0)=\urho_t$ is a semigroup of
  continuous maps in $\PlusMeasuresTwo{\R,\frak m}$ satisfying the
  stability property
  \begin{equation}
    \label{eq:15}
    W_2(\sfS_t(\urho_0),\sfS_t(\eta_0))\le \rme^{-\lambda t}W_2(\urho_0,\eta_0).
  \end{equation}
  If moreover $\rho_0\le \eta_0$ then $\sfS_t(\rho_0)\le
  \sfS_t(\eta_0)$ for every $t\ge0$.

\end{theorem}

  \begin{remark}[Singularities]
    \label{rem:singular}
    Recalling the definition \eqref{eq:5} of $\cI$, the regularization
    estimate \eqref{eq:18} shows that the solution given by Theorem
    \ref{thm:main1} satisfies $\urho_t\in\CPlusMeasuresTwo{\R,\mass}$
    for every $t>0$.
  \end{remark}

%\GGr{A direct consequence of the stability property \eqref{eq:15} is the following}
% \begin{m}
%   \begin{corollary}            \label{cor:cd}
% Let $\urho_{t}^n$ be a sequence of solutions arising from compactly
%   supported and bounded initial densities $\urho_0^n$ converging to $\urho_0$ in
%   $\PlusMeasuresTwo{\R,\frak m}$. Then $\urho^n_t\to \urho_t$ in
%   $\PlusMeasuresTwo{\R,\frak m}$ for every $t\ge0$.
% \end{corollary}
% \end{m}
{In the case when $V'$ is Lipschitz, the stability property
  \eqref{eq:15} and a simple regularization of the initial datum
  show that Wasserstein solutions are the limit of locally bounded
  solutions satisfying \eqref{eq:9}.}
Another way to see that Definition \ref{def:measuresolution} provides the right notion of solution involves a classical viscous regularization of
\eqref{eq:rho} {combined with a suitable regularization of the potential $V$}.
Given a small parameter $\eps>0$ let us consider the perturbed nonlinear functions
\begin{equation}
  \label{eq:16}
  \beta^\eps(r):=\beta(r)+\eps r,\quad r\in [0,+\infty),
\end{equation}
{and a family $V^\eps$ of smooth and Lipschitz potentials such that
  \begin{subequations}
    \begin{gather}
    \label{eq:82}
     V^\eps(x)\le V(x)+A|x|^2\quad
      \lambda\le (V^\eps)''(x)\le \sup_\R V''\qquad
      \text{for every }x\in\R,\\
      \label{eq:85}(V^\eps)^{(h)} \to V^{(h)}
      \quad\text{as }\eps\downarrow0
    \quad\text{uniformly on compact sets of }\R,
    \quad h=0,1,2,\\
    \label{eq:86}
    \liminf_{|x|\to\infty}\frac{V^\eps(x)}{|x|^2}\ge0\quad\text{uniformly
      with respect to $\eps$.}
 \end{gather}
  \end{subequations}
}
For every $\urho^\eps_{0}\in \PlusMeasuresTwo{\R,\frak m}$ we consider the problem
\begin{equation}\label{eq:17}
  \partial_t \urho^\eps-\partial_x\big(\partial_x\beta^\eps(u^\eps)+(V^\eps)' \urho^\eps\big)=0,
  \quad {\rm in }\ (0,+\infty)\times  \R;
    \quad
    \urho^\eps(0,\cdot)=\urho^\eps_0,\quad \text{in }\R,
\end{equation}
the associated energy functional
\begin{equation}
  \cE^\eps(\urho)=
  \begin{cases}
    \cE(\urho) + \eps\displaystyle \int_\R u\log u \,\d x & \text{if } \urho=u\Leb{1} \ll \Leb{1}\\
    +\infty & \text{if } \urho^\perp \not= 0
  \end{cases}
  \quad
  {\cV^\eps(\rho):=\int_\R V^\eps(x)\,\d\rho,\quad
  \cF^\eps=\cE^\eps+\cV^\eps,}
\end{equation}
and the corresponding Fisher-dissipation
\begin{equation}
  \label{eq:46bis}
  \cI^\eps(\urho):=\int_\R\Big|\frac{\partial_x\beta^\eps(u)}u+(V^\eps)'\Big|^2u\,\d x \qquad\text{if }\urho=u\Leb1,\
  \beta^\eps(u)\in W^{1,1}_{\rm loc}(\R).
\end{equation}
As usual $\cI^\eps(\urho)=+\infty$ if $u\not\in W^{1,1}_{\rm loc}(\R)$
or
%\GGr{even}
$\urho\not\ll\Leb 1$.
\begin{theorem}[Convergence of viscous regularizations]
  \label{thm:main2}
  For every {$\urho^\eps_{0}=u^\eps_0\Leb1\in \PlusMeasuresTwo{\R,\frak
    m}$
  with $u^\eps_0\in C^1_c(\R)$},
  there exists a unique {smooth} solution $\urho^\eps=u^\eps\Leb 1\in
  C^0([0,+\infty); \PlusMeasuresTwo{\R,\frak m})$
  of problem \eqref{eq:17} {satisfying $\cI^\eps(\rho^\eps)\in L^1_{loc}(0,+\infty)$.}
%
  %\GGr{belonging to $ $ and such that $\beta^\eps(u^\eps)\in L^1_{\rm
%      loc}((0,+\infty); W^{1,1}_{\rm loc}(\R))$.}
  Moreover \eqref{eq:18}, \eqref{eq:19}, and \eqref{eq:13} hold with $\cF, \cI$ replaced by $\cF^\eps, \cI^\eps$, respectively.

  If $\urho^\eps_0\to \urho_0$ in $\PlusMeasuresTwo{\R,\frak m}$
  and
  $\sup_\eps \cF^\eps(\urho^\eps)<+\infty$, then $\urho^\eps_t$ converges in
  $\PlusMeasuresTwo{\R,\frak m}$ to the unique Wasserstein solution $\urho_t$ of
  problem \eqref{eq:rho}
  %\GGr{given by Theorem \ref{thm:main1},}
   as $\eps\downarrow0$ for every $t>0$.
  Moreover
%  \fcolorbox{red}{green}{$\beta^\eps(u^\eps_t)\to \beta(u_t)$ uniformly on
%  compact sets in $\R$ and
  $u^\eps_t\to u_t$ uniformly on compact sets
  of $\Dom{u_t}$ for every $t>0$.
\end{theorem}
The proofs of Theorems \ref{thm:main1} and \ref{thm:main2} {take
  advantage of}
%\GGr{based on}
the theory of gradient flows of convex functionals with respect to the
Wasserstein distance \cite{ags} and  will be given in Section \ref{sec:gf}.

    \begin{remark}[Non smooth potentials]
    \label{rem:nonsmooth}
    Theorems \ref{thm:main1} and \ref{thm:main2} are still true in the case when $V$ is a
    general $\lambda$-convex function, i.e.\ the condition
    \eqref{crescita-quadratica} on the lower bound on $V''$ (which we
    assumed
    for the sake of simplicity) is
    replaced by
    \begin{equation}
      \label{eq:79}
      x\mapsto V(x)-\frac\lambda 2x^2\quad\text{is convex in }\R.
    \end{equation}
    \eqref{eq:79} implies that $V$ is differentiable $\Leb 1$-almost
    everywhere, so that the first occurence of $V'$ in the definition
    \eqref{eq:5} of $\cI$
    still makes sense as it is integrated with respect to $\Leb
    1$. The second integral term in \eqref{eq:5} should be replaced by
    \begin{equation}
      \label{eq:81}
      \int_\R |\partial^\circ  V(x)|^2\,\d\rho^\perp(x)
    \end{equation}
    where $\partial^\circ V(x)$ denotes the element of minimal norm in
    the (non empty) Frechet subdifferential $\partial V$ of $V$.
  \end{remark}

\subsection{Propagation of singularities.}
\label{subsec:propagation}

  In this section we want to study the evolution of the singular
  part $\rho_t^\perp$ of the Wasserstein solution $\rho_t$ to
  \eqref{eq:rho}.
  By Remark \ref{rem:singular} we know that $\rho_t=u_t\Leb1+\rho_t^\perp\in
  \CPlusMeasuresTwo\R$ for every $t>0$, so that the support of
  $\rho_t^\perp$ coincides with
  the set where the (continuous representative of the) density $u_t$
  takes the value $+\infty$. We thus call

%\GGr{Let us suppose that $\urho_0=u_0\Leb 1+\urho_0^\perp\in
%\CPlusMeasuresTwo{\R}$
%and let}
$\CDom{u_t}:=\R\setminus\Dom{u_t}=\big\{x\in
\R:u_t(x)=+\infty\big\} $
{and we will show that the evolution of $\CDom{u_t}$ follows the
flow generated by $-V'$.}

Let us first introduce the evolution semigroup $\sfX$ on $\R$ generated by $-V'$, thus satisfying
\begin{equation}
  \label{eq:20}
  \frac \d{\d t}\sfX_t(x)=-V'(\sfX_t(x)),\quad
  \sfX_0(x)=x\quad\text{for every }x\in \R.
\end{equation}
Since $V'$ is of class $C^1$ and, by \eqref{crescita-quadratica},
$$\big(V'(x)-V'(y)\big)(x-y)\ge\lambda|x-y|^2\quad \text{for every }x,y\in \R,$$
$\sfX_t$ is a family of diffeomorphisms {mapping $\R$ onto the open
  set $\sfR_t:=\sfX_t(\R)$.}
We set
\begin{equation}
  \label{eq:21}
  {\mathsf J}_t:=\sfX_t\big(\CDom{u_0}\big),\quad
  {\mathsf D}_t:=\sfX_t\big(\Dom{u_0}\big),\quad t\ge0,
\end{equation}
and we notice that $\mathsf J_t={\sfR_t}\setminus \mathsf D_t$ is a closed
subset of {$\sfR_t$}, since $\mathsf D_t$ is open.
% \footnote{Non \`e detto che $\sfR_t=\R$, perch\'e non assumiamo
%   nulla circa la crescita di $V'$}

{If $\sigma\in \PlusMeasures\R$, the push-forward $(\sfX_t)_\#\sigma$
  through
  $\sfX_t$ is the Borel measure defined by}
 \[
 (\sfX_t)_\# \sigma(A) := \sigma\big(\sfX_t^{-1}(A)\big)\quad\text{for each Borel set $A \subset \R$}.
 \]
 \begin{theorem}[Propagation of singularities]
  \label{thm:main3}
  If $\urho_0\in\CPlusMeasuresTwo{\R}$ and
  $\urho_t=u_t\Leb 1+\urho_t^\perp\in \CPlusMeasuresTwo\R$ is the
  unique {Wasserstein} solution of \eqref{eq:rho},
  then
  \begin{equation}
    \label{eq:68}
    \partial_t \urho_t^\perp-\partial_x\big(\urho_t^\perp\,V'\big)\le
    0\quad\text{in the sense of distributions,}
    \quad
    {\lim_{t\downarrow0}\rho_t^\perp\le \rho_0^\perp
    \quad\text{weakly in }\PlusMeasures\R.}
  \end{equation}
  In particular
  \begin{equation}
    \label{eq:22}
    \CDom{u_t}\subset \mathsf J_t, \quad
    \urho_t^\perp\le (\sfX_t)_\#\urho_0^\perp,\quad
    \text{for every }t\ge0,
  \end{equation}
  so that for every Borel set $A\subset \R$
  \begin{equation}\label{eq:23}
    \urho_t^\perp(A)\le \urho_0^\perp \big(\sfX_t^{-1}(A)\big).
  \end{equation}
  In particular $\urho_t^\perp$ is concentrated in $\sfX_t(\CDom{u_0})$
  and $u_t$ is finite in $\sfX_t(\Dom{u_0})$.
\end{theorem}
The proof of Theorem \ref{thm:main3} will be carried out in Section \ref{sec:estimates}.

The case when $\urho_0^\perp=\sum_{j=1}^N\alpha_j\delta_{x_j}$ with $x_1<x_2<\cdots<x_N$ and $\alpha_j>0$ is of particular interest. In this case,
from Theorem \ref{thm:main3} we deduce that $\urho_t=u_t\Leb1+\urho_t^\perp$ with
\begin{equation}
  \label{eq:24}
  \urho_t^\perp=\sum_{j=1}^N \alpha_j(t)\delta_{x_j(t)},\quad
  x_j(t)=\sfX_t(x_j),
\end{equation}
where $\alpha_j:[0,+\infty)\to [0,+\infty)$ { is nonincreasing.}

  Theorem \ref{thm:main3} can be equivalently formulated in terms of
  the density $u_t$ of the regular part of $\rho_t$:
  \begin{corollary}[The regular part is a supersolution]
    If $\rho_t=u_t\Leb1+\rho_t^\perp\in \CPlusMeasuresTwo\R$ is a
    Wasserstein solution to \eqref{eq:rho} then $u_t$ is a
    supersolution of \eqref{eq:rho}, i.e.
    \begin{equation}
      \label{eq:71}
      \partial_t
      u-\partial_x\big(\partial_x\beta(u)+V'\,u\big)\ge0\quad\text{in
        the sense of distributions in }(0,+\infty)\times\R.
    \end{equation}
  \end{corollary}

\subsection{Minimizers of the energy functional and critical mass.}
\label{subsec:minimizers}

  In this section we will assume that the potential $V$ satisfies the
  coercivity condition
  \begin{equation}
    \label{eq:95}
    \lim_{|x|\to+\infty}V(x)=+\infty
    \quad \text{and we set}\quad
    V_{\rm min}:= \min_\R V,\quad
    Q:=\big\{x\in \R:V(x)=V_{\rm min}\big\},
    \tag{2.coer}
  \end{equation}
and we study the minimizers of the functional $\cF$,
which are particular steady states of equation \eqref{eq:rho}.
% Theorem \ref{thm:main3} guarantees that the mass of the singular part cannot increase  during the evolution. Despite this fact, the mass
% of the singular part can be created exactly at $t=+\infty$. Indeed, as
% we already observed, under suitable conditions on $\beta$ and $V$, there
% exists a critical mass $\mass_{\rm c}<+\infty$. If  $\mass:=\urho_0(\R) > \mass_{\rm c}$, as $t\uparrow +\infty$, the solution $\urho_t$ converges in
% $\PlusMeasuresTwo{\R,\mass}$ (under the assumption of convexity of $V$) to a stationary state with a singular part of mass $\mass- \mass_{\rm
% c}>0$, independently of the presence of a singular part in $\urho_0$ (see Theorem \ref{thm:main6}).
%
% In Theorem \ref{thm:main4} we characterize the minimizers
% \GGr{in $\CPlusMeasuresTwo{\R,\mass}$}
% of the functional $\cF$ and we investigate their connection with the steady states of equation \eqref{eq:rho}, namely the measures $\urho$ such that $\cI(\urho)=0$.  Next, in Theorems \ref{thm:main5} and \ref{thm:main6} we describe the asymptotic behaviour of the solution for large times.
%
%\begin{displaymath}
  The structure of the minimizers of $\cF$ is governed by two critical constants
  and two functions, with their inverses.
  The first function is $r\mapsto -E'(r)$: it is a decreasing
  homeomorphism between $(0,+\infty)$ and the interval $(0,\dV)$, which
  can be characterized by
  the constant
  \begin{displaymath}
%    \label{eq:94}
    \dV:=-\lim_{x\to 0^+}E'(x)=\int_0^{+\infty}\frac{\beta'(s)}s\,\d
    s\in (0,+\infty].
  \end{displaymath}
  Notice that $\dV$ is finite if and only if $s\mapsto \beta'(s)/s$ is integrable
  in a right neighborhood of $0$.
  We can thus consider
  the pseudo-inverse function
  $H:(0,+\infty)\to [0,+\infty)$ defined by
  $$
  H(v)=\begin{cases}(E')^{-1}(-v)& \text{if }v\in(0,\dV)\\
    0& \text{if $\dV<+\infty$ and }v\in [\dV,+\infty)
  \end{cases}
$$
which is decreasing in the interval $(0,\dV)$.

The second function is
% From the convexity of $E$ it follows that $H$ is nonincreasing.
% Notice that, by \eqref{eq:121}, $E'_0>-\infty$ if and only if
% the function $r\mapsto \beta'(r)/r$ is integrable near $0$.
% In this paragraph, we assume that the potential $V$ satisfies the coercivity conditions
% \begin{equation}
%   \label{eq:25}
%   \begin{gathered}
%     \lim_{|x|\to+\infty}V(x)=+\infty,\quad
%     \int_{\{x:V(x)>1\}} x^2 H(V(x))\,\d x<+\infty.
%   \end{gathered}
% \end{equation}
% Setting
% $$ V_{\rm min} = \min_\R V, \qquad Q:=\big\{x\in \R: V(x)=V_{\rm min}\big\},$$
% we also assume that
% \begin{equation}             \label{Qdisc}
% Q \text{ is discrete.}
% \end{equation}
\begin{displaymath}
%  \label{eq:28a}
  M_\R(v):=\int_\R H(V(x)-v)\,\d x,\quad
  v\le V_{\rm min}.
\end{displaymath}
In order to avoid a degenerate situation, we will assume that $V$
satisfies the integrability condition
% \footnote{questa serve per garantire che il candidato minimimo abbia
%   massa finita, propriet\`a che utilizziamo anche per
%   dimostrare l'unicit\`a del minimo. Non so se si pu\`o togliere il
%   $V_{\rm min}$ all'interno di $H$ perch\'e non conosciamo la crescita
% di $H$. Nella caratterizzazione del minimo non si usa mai la finitezza
% del momento secondo; questa viene aggiunta dalla \eqref{eq:97}}
\begin{equation}
  \label{eq:96}
  \int_{\R\setminus \tilde Q}H(V(x)-V_{\rm min})\,\d
  x<+\infty,\quad\text{for some bounded open neighborhood $\tilde Q$
    of $Q$}.
  \tag{2.int}
\end{equation}
\eqref{eq:96} yields $M_\R(v)<+\infty$ for every $v<V_{\rm min}$
so that
$M_\R$ is an increasing homeomorphism between $(V_{\rm min}-\dV,V_{\rm
  min})$ and the interval $(0,\mass_{\rm c})$,
where the critical mass $\mass_{\rm c}$ is defined by
\begin{equation}
  \label{eq:28}
  \mass_{\rmc}:=\lim_{v\uparrow V_{\rm min}}M_\R(v)=
  \int_\R H(V(x)-V_{\rm min})\,\d x\in (0,+\infty].
\end{equation}
If $M^{-1}:(0,\frak m_\rmc)\to (V_{\rm min}-\dV,V_{\rm min})$ denotes the
inverse map of $M$, we eventually set
\begin{equation}\label{eq:31}
  \frak v:=\begin{cases}
    M_\R^{-1}(\mass) & \text{if } \mass < \mass_\rmc \\
    V_{\rm min} & \text{if } \mass\geq \mass_\rmc.
    \end{cases}
\end{equation}
\begin{theorem}[Characterization of minimizers]
  \label{thm:main4}
  If $V$ satisfies \eqref{eq:95} then
  $\cF$ attains its minimum on $\PlusMeasures{\R,\mass}$.
  If  $V$ also satisfies \eqref{eq:96} then
  a measure $\rho\in \PlusMeasures{\R,\mass}$ is a minimizer of
  $\cF$ if and only if it belongs to $\CPlusMeasures{\R,\mass}$ and
  its decomposition $\urho_{\rm min}=u_{\rm min}\Leb 1+\urho_{\rm
    min}^\perp$ satisfies
 \begin{equation}\label{eq:30}
    \begin{aligned}
      u_{\rm min}(x)=H(V(x)-\frak v),\quad
      \urho_{\rm min}^\perp(\R\setminus Q)=0,\quad
    \urho_{\rm min}^\perp (Q)=(\frak m - \frak m_\rmc)^+.
  \end{aligned}
  \end{equation}
  $\rho_{\rm min}$ belongs to $\PlusMeasuresTwo{\R,\mass}$ if $V$
  satisfies
  the condition (stronger than \eqref{eq:96})
  \begin{equation}
    \label{eq:97}
    \int_{\R\setminus \tilde Q}|x|^2\,H(V(x)-V_{\rm min})\,\d
  x<+\infty,\quad\text{for some bounded open neighborhood $\tilde Q$ of $Q$}.
  \end{equation}

\end{theorem}

\begin{remark}
  \label{rem:examples}
\
\begin{itemize}
\item
%  \GGr{It follows immediately from the previous Theorem that}
  In the case when $\frak m \le \frak m_{\rmc},$
  the minimizer $\rho_{\rm min}=u_{\rm min}\Leb 1$
  is unique and $\rho_{\rm min}^\perp=0$.
  If $\mass<\mass_{\rm c}$, $u_{\rm min}$ is bounded, whereas
  if $\mass=\mass_{\rm c}$, $u_{\rm min}(x)=+\infty$ for every $x\in Q$.
  Last, if $\frak m > \frak m_{\rmc}$
  the minimizer has a nontrivial singular part and
  it is unique only when $Q$ is a singleton.
\item
As already pointed out, the existence of the critical mass $\mass_{\rm
  c}<+\infty$
depends on the behavior of the $\beta(r)$ for large values of $r$
{and
on the local behaviour of $V$ near $Q$}.
%\GGr{More precisely, it depends both on the behavior of the vanishing of the function $r\mapsto E'(r)$ as $r\to+\infty$, and on the
%behavior of the vanishing of $x\mapsto V(x)-V_{\rm min}$ as $x\to x_0\in Q$.}
\item
  {If $\dV<+\infty$ %and $\liminf_{|x|\to\infty}V(x)>\frak v+\dV$
    then
  the support of $\urho_{\rm min}$ is compact and it is contained in
  the sublevel of $V$
  $\{ x\in\R:V(x)\le \frak v+\dV\}$.}
\item {If $Q$ is an interval (in particular if $V$ is convex) then the minimizer of $\cF$ is unique.
    This property is always true when $\mass_\rmc=+\infty$;
    when $\mass_\rmc<+\infty$, the fact that $Q$ is a closed interval
    and \eqref{eq:28} show that $Q$ is a singleton.}
\end{itemize}
\end{remark}

\subsection{Stationary solutions}
\label{subsec:stationary}

  In this section we will study the stationary Wasserstein solutions
  of \eqref{eq:rho},
  i.e.\ constant measures $\rho\in \PlusMeasuresTwo{\R}$ which
  solve \eqref{eq:rho}.
  As a starting point, we observe that steady states can be
  characterized as measures
  with vanishing Fisher dissipation.
  \begin{theorem}
    \label{thm:Fisher=0}
    A measure $\rho\in \PlusMeasuresTwo{\R,\mass}$ is a stationary
    Wasserstein
    solution of \eqref{eq:rho} iff $\rho\in
    \CPlusMeasuresTwo{\R,\mass}$ and
    $\cI(\rho)=0$.
  \end{theorem}
  Of course, any minimizer $\rho$ of $\cF$ in
  $\PlusMeasuresTwo{\R,\mass}$ satisfies $\cI(\rho)=0$ and it is a
  stationary solution,
  but in general one can expect that other stationary solutions exist.
  Their structure depends in a crucial way on
  $\dV$; the simplest case is when $\dV=+\infty$.
  \begin{theorem}[Characterization of stationary measures I]
    \label{thm:statI}
    Let us suppose that $V$ satisfies
    \eqref{eq:95} and \eqref{eq:96}.
    If $\dV=+\infty$ then for every $\rho\in \PlusMeasures{\R,\mass}$
      \begin{equation}\label{nulldissipation}
        \text{$\cI(\urho)=0 \quad
          \Leftrightarrow\quad \urho$ is a minimizer for
          $\cF$ in $\PlusMeasures{\R,\mass}$.}
  \end{equation}
  In particular, a measure $\rho\in
  \PlusMeasuresTwo{\R,\mass}$ is
  a stationary solution if and only if it is a minimizer of $\cF$.
  \end{theorem}
  The case when $\dV<+\infty$ is more complicated and requires some
  preliminary definition.
  \begin{definition}
    \label{def:adm_interval}
    Let us suppose that $\dV<+\infty$.
    We say that a bounded open interval $I=(a,b)\subset\R$ is
    an admissible local sublevel of $V$ if
    \begin{equation}
      \label{eq:98}
      V(a)=V(b),\quad
      \frak v_I:=V(a)-\dV\le V(x)< V(a)\quad\text{for every $x\in
        (a,b)$},
   \end{equation}
   and
   \begin{equation}
     \label{eq:99}
     M_I:=\int_a^b H\big(V(x)-\frak v_I\big)\,\d x<+\infty.
   \end{equation}
   We set $Q_I:=\big\{x\in I
   :V(x)=\min_I V\big\}$.
\end{definition}
Notice that $Q_I$ is not empty iff
\begin{gather}
  \label{eq:98bis}
  \frak v_I=V(a)-\dV=\min_{I}V.
\end{gather}
If $Q_I$ is empty, i.e.\ $\frak v_I<\min_I V$,
then condition \eqref{eq:99} is always satisfied.

If $u:\R\to [0,+\infty]$ is a continuous map, we set
\begin{equation}
  \label{eq:104}
  \begin{aligned}
    \Pos u:=&\big\{x\in \R:u(x)>0\big\},\\
    \Conn u:=&\text{the
      collection of all the connected components of $\Pos u$.}
  \end{aligned}
\end{equation}
\begin{theorem}[Characterization of stationary measures II]
  \label{thm:main_stationary}
  % Let us assume \eqref{eq:25} and \eqref{Qdisc}. The following assertions hold.
  Let us suppose that $V$ satisfies
  \eqref{eq:95} and \eqref{eq:96}.
  If $\dV<+\infty$ a measure
  $\urho=u\Leb 1+\rho^\perp\in \CPlusMeasures{\R,\mass}$ satisfies
  $\cI(\rho)=0$ % is a
  % stationary Wasserstein solution for \eqref{eq:rho}
  if and only if it
  satisfies the following
  three conditions:
  \begin{enumerate}
  \item
    All the connected components in $\Conn u$ of the open set
    $\Pos u$ are admissible local
    sublevels of $V$
    according to Definition \ref{def:adm_interval}.
    \item
      \begin{equation}
      \label{eq:103}
      u\restr{I}=H(V(x)-\frak v_I)\quad\text{for every }I\in \Conn u.
    \end{equation}
    \item
      If $Q(u):=\bigcup_{I\in
        \Conn u}Q_I$
    \begin{equation}
      \label{eq:84}
      \text{$\rho^\perp$ is concentrated on $Q(u)$,
        and}
      \quad\frak m=\sum_{I\in \Conn u}M_I
      +\rho^\perp(\R).
    \end{equation}
  \end{enumerate}
\end{theorem}
\begin{corollary}
  \label{cor:obvious}
  If $V$ satisfies
  \eqref{eq:95}, \eqref{eq:96},
  and
  \begin{equation}
    \text{the set $Q$ of \eqref{eq:95} is an interval $[q_-,q_+]$, $V'\ge0$ in
      $(q_+,+\infty)$, and $V'\le 0$ in $(-\infty,q_-),$}
\label{eq:110}
\end{equation}
(\eqref{eq:110} is always satisfied if $V$ is convex),
then
  \eqref{nulldissipation} holds
  and there exists a unique stationary measure in
  $\PlusMeasuresTwo{\R,\mass}$ which coincides with the unique
  minimizer of $\cF$ in $\PlusMeasuresTwo{\R,\mass}$.
\end{corollary}
\begin{remark}
  \label{rem:converse_obvious}
  It is possible to prove a converse form of Corollary
  \ref{cor:obvious}:
  if $\dV<+\infty$ and
  for every value of $\mass>0$ there exists a unique steady state
  in $\PlusMeasuresTwo{\R,\mass}$ then $V$ satisfies \eqref{eq:110}.
\end{remark}

%The next two examples illustrate the various situations % and the necessity of  assumption \eqref{hp:sing} for the validity of the implication \eqref{nulldissipation}.

\begin{example}\label{ex:1}
Let us choose $\beta(r)=\arctan r$, so that $E(r)=\displaystyle r\log\left(\frac{ r}{\sqrt {1+r^2}}\right)-\arctan r$ and $E'(r)=\displaystyle
\log\left(\frac{ r}{\sqrt {1+r^2}}\right)$. Notice that $\dV=+\infty$. One can compute explicitly $\displaystyle
H(v)=\frac{\rme^{-v}}{\sqrt{1-\rme^{-2v}}}$, for $v>0$. If the potential is $V(x)=|x|^\alpha$, with $\alpha>1$, the critical mass is defined by
$\displaystyle \mass_{\rm c}=\int_\R \frac{\rme^{-|x|^\alpha}}{\sqrt{1-\rme^{-2|x|^\alpha}}}\,\d x.$ It follows that $\mass_{\rmc}<+\infty$ if and only if $\alpha<2$.

We find that
\begin{equation}
u_{\rm min}(x)=\frac{\rme^{-|x|^\alpha+\frak
    v}}{\sqrt{1-\rme^{-2|x|^\alpha+2\frak v}}}.
\end{equation}
If $\alpha \ge 2$, for every value of the mass $\mass$, the unique minimum point, which is also the unique stationary solution, can not have a
singular part, and it is bounded and positive. The same situation occurs when $\alpha<2$ and $\mass < \mass_\rmc$. If $\alpha<2$ and $\mass = \mass_\rmc$, then the unique stationary state is infinite
at $x=0$ but without a singular part, whereas for $\mass >\mass_\rmc$ the singular part is $\urho^\perp_{\rm min}=(\mass-\mass_\rmc)\delta_0$.
\end{example}

\begin{example}\label{ex:2}
Let us choose  $\beta(r)=\frac{r^2}{1+r^2}$. Then $E(r)=-r\arctan (1/r)$ and $E'(r)=\frac{r}{1+r^2}-\arctan (1/r)$. In this case
$\dV=\frac \pi 2$.

Let us observe that $E'(r)$ has the same behavior of $r\mapsto-1/r^3$ as $r\to +\infty$. Therefore $H(v)$ has the same behavior of $v\mapsto
v^{-1/3}$ for $v\to 0^+$. Considering again the potential $V(x)=|x|^\alpha$, with $\alpha>1$, it follows that $\mass_{\rmc}<+\infty$ if and only
if $\alpha<3$.

The support of the unique stationary state
is $\{x\in\R: |x|\le (\frak v+\pi/2)^{1/\alpha}\}$
and it is compact for every value of $\mass$ and $\alpha$.

Finally we show %\GGr{that it
%  may exist}
a measure $\urho\in\CPlusMeasuresTwo{\R,\mass}$ satisfying
$\cI(\urho)=0$ which is not of the form \eqref{eq:30}.

To this aim we consider the double well potential $V(x)=\pi(x-1)^2(x+1)^2$. Let $\mass>\mass_\rmc$. Defining $u(x)=H(V(x))$ for $x>0$ and $u(x)=0$
for $x\leq 0$, we observe that $u$ is continuous on $\R$ with values in $[0,+\infty]$ and $\int_\R u(x)\,\d x = \mass_\rmc/2$. Consequently, the
measure $\urho= u\Leb{1}+(\mass-\mass_\rmc/2)\delta_1$ belongs to $\CPlusMeasuresTwo{\R,\mass}$, satisfies $\cI(\urho)=0$ but is not of
the form \eqref{eq:30}.

We can construct a similar example when $V$ has a local minimum greater than $V_{\rm min}$.  For instance we can consider a potential $V$ defined
by $V(x)=2\pi(x+1)^2+1$ for $x<-1/2$, $V(x)=2\pi(x-1)^2$ for $x>1/2$ and suitably defined in $[-1/2,1/2]$ in order to satisfy the condition $V(x)>\pi/2$ and the
$\lambda$-convexity assumption. Then the support of $\urho_{\rm min}$ is contained in $[-3/2,-1/2]\cup[1/2,3/2]$. Let us define
$u(x)=H(V(x)-1)$ for $x<0$ and $u(x)=0$ for $x\geq 0$, and $\urho= u\Leb{1}+(\mass-\tilde\mass_\rmc)\delta_{-1}$, where $\tilde\mass_\rmc:=
\mass_\rmc - \int_{-\infty}^0H(V(x))\,\d x$. Then $\urho\in\CPlusMeasuresTwo{\R,\mass}$, $\cI(\urho)=0$ but $\urho$ is not of the form
\eqref{eq:30} and it is not a minimizer of $\cF$.
\end{example}

%\begin{theorem}
%  \label{thm:main4}
%  If \eqref{eq:25} is satisfied then the energy functional $\cF$ is
%  bounded from below and admits a minimum in $\CPlusMeasuresTwo{\R,\mass}$.
%  If the zero set of $V'$ is discrete, then $\cI(\urho)=0$
%  if and only if $\urho$ is a minimizer of $\cF$.\\
%  The minimizers of $\cF$ are of the form
%  \begin{equation}
%    \label{eq:30}
%    \urho_{\rm min}=u_{\rm min}\Leb 1+\urho_{\rm min}^\perp,\quad
%    u_{\rm min}(x)=H(V(x)-y_{\frak m}),\quad
%    \urho_{\rm min}^\perp(\R\setminus Q)=0,\quad
%    \urho_{\rm min}^\perp (Q)=(\frak m - \frak m_\rmc)^+,
%  \end{equation}
%  where
%  \begin{equation}\label{eq:31}
%    y_{\frak m}:=\begin{cases}
%    M^{-1}(\mass) & \text{if } \frak m < \frak m_\rmc \\
%    V_{\rm min} & \text{if } \frak m \geq \frak m_\rmc
%    \end{cases}.
%  \end{equation}
%  In particular, if
%  $\frak m \le \frak m_{\rmc},$
%  then the minimizer is unique and there is no singular part,
%  whereas in the case $\frak m > \frak m_{\rmc}$
%  the minimizers possesses necessarily a singular part and, in general are not unique.\\
%  If $Q=\{x_{\rm min}\}$ is a singleton
%  (i.e.\ $V$ has a unique minimizer),
%  then $\urho_{\rm min}=u_{\rm min}\Leb 1+(\frak m - \frak m_\rmc)^+\delta_{x_{\rm min}}$ is the
%  unique minimizer of  $\cF$.
%\end{theorem}

\begin{remark}
We point out that the case $\dV=+\infty$ reveals some analogies with a diffusion which is linear near to $0$. In this case we
have the immediate strict positivity of the solution also starting from compactly supported initial data.

On the contrary, the case $\dV<+\infty$ corresponds to a slow diffusion near to $0$. In this case, starting from compactly supported
initial data the solution could remain compactly supported for all
time and it may happen that
as $t\to+\infty$ the solution converges to a steady state which is not
a global minimum of $\cF$.
\end{remark}

\begin{remark}[Examples of singular solutions]
  \label{rem:singularex}
  Let $\rho:=u\Leb 1+\rho^\perp\in \CPlusMeasuresTwo{\R,\mass}$ be a
  steady state
  of \eqref{eq:rho} with
  $\rho^\perp\neq 0$: e.g., one can consider the case when
  $\mass_\rmc<+\infty$ and take a minimizer of $\cF$ with
  $\mass>\mass_\rmc$.
  If $\tilde\rho_0=\tilde u\Leb 1+\rho^\perp\in \CPlusMeasuresTwo{R,\mass}$ with
  $\tilde u\ge u$ then the comparison principle shows that the
  Wasserstein
  solution $\tilde\rho_t$ of \eqref{eq:rho} with initial datum $\tilde
  \rho_0$
  is singular and its singular part is $\rho^\perp$ for every $t\ge0$.
\end{remark}

\subsection{Asymptotic behaviour}
\label{subsec:asymptotic}
{Let us first considering the case of a convex potential $V$.}
Here we can
apply the general results about the asymptotic behavior
for displacement convex functionals (see \cite{ags}).

  Moreover, as we observed in Remark \ref{rem:examples}, the specific
  form of the functional $\cF$ yields that it has only one minimizer
  $\urho_{\rm min}$ in each class $\CPlusMeasuresTwo{\R,\mass}$ which
  is also the unique steady state by Theorem \ref{thm:statI} and
  Corollary \ref{cor:obvious}: the study of the asymptotic behaviour
  is therefore greatly simplified.

\begin{theorem}[Asymptotic behavior I: the convex case]\label{thm:main6}

  Let us assume that the potential $V$ is convex (i.e.\
    \eqref{crescita-quadratica} is satisfied with $\lambda=0$)
    and
    satisfies \eqref{eq:95} and \eqref{eq:97},
    and
  let $\rho_{\rm min}$ be the unique minimizer of $\cF$ in
  $\PlusMeasuresTwo{\R,\mass}$.
  If $\rho$ is a Wasserstein solution to \eqref{eq:rho} in
  $\PlusMeasuresTwo{\R,\mass}$
  then $\rho_t$ weakly converges to $\rho_{\rm min}$ as $t\to+\infty$
  in the duality
  with continuous and bounded functions. Moreover, for every $t\in (0,+\infty)$
  \begin{equation}
    \label{eq:105}
 %  \begin{array}{l}
 \displaystyle  \cF(\urho_t)-\cF(\urho_{\rm min}) \leq \frac{W^2_2(\urho_0,\urho_{\rm min})}{2t} , \quad
 \displaystyle \cI(\urho_t) \leq \frac{W^2_2(\urho_0,\urho_{\rm
     min})}{t^2}.
 %  \displaystyle  \lim_{t\uparrow+\infty}W_2(\urho_t,\urho_{\rm min})=0.
 % \end{array}
  \end{equation}

  If the potential $V$ also satisfies \eqref{crescita-quadratica} with
  $\lambda>0$, then
  for every $t>0$ we have the exponential estimates
  \begin{gather}
      W_2(\urho_t,\urho_{\rm min}) \leq \rme^{-\lambda t} W_2(\urho_0,\urho_{\rm min}),\\
\cF(\urho_t)-\cF(\urho_{\rm min}) \leq \rme^{-2\lambda t} \big(\cF(\urho_0)-\cF(\urho_{\rm min})\big),\quad
     \cI(\urho_t) \leq \rme^{-\lambda t}\frac{W^2_2(\urho_0,\urho_{\rm min})}{t^2}.
\end{gather}
\end{theorem}

  The last result concerns more general potentials $V$: a simple
  characterization
  of the asymptotic behaviour of a Wasserstein solution is possible
  only when there exists a unique steady state for \eqref{eq:rho}
  (which therefore coincides with the minimizer of $\cF$):
  this is the case when $\dV=+\infty$ and
  $V$ satisfies \eqref{eq:95} and \eqref{eq:97}, or when $\dV<+\infty$ and $V$
  satisfies
  the conditions of Corollary \ref{cor:obvious}.

\begin{theorem}[Asymptotic behavior II]
  \label{thm:main5}
  Let us suppose that $V$ satisfies \eqref{eq:95} and \eqref{eq:96}
  and
  let us assume that there exists a unique steady state $\bar\rho\in
  \PlusMeasures{\R,\mass}$
  with $\cI(\bar\rho)=0$ ($\bar\rho$ is also the unique minimizer of
  $\cF$ in $\PlusMeasures{\R,\mass}$).
  If $\rho$ is a Wasserstein solution to \eqref{eq:rho} in
  $\PlusMeasuresTwo{\R,\mass}$ then
  \begin{equation}
    \label{eq:109}
    \rho_t\weakto\bar\rho\quad\text{weakly as }t\to+\infty,\quad
    \lim_{t\uparrow+\infty}\cI(\rho_t)=0.
  \end{equation}
  In particular the continuous density $u_t$ converges to $\bar u$
  uniformly on the compact sets of $\Dom {\bar u}$;
  if moreover the support of $\rho_0^\perp$ is compact and
  $\mass<\mass_\rmc$,
  then there exists a finite time $T>0$
  such that $\rho_t\ll\Leb 1$ for every $t\ge T$.
% \\
%   If either
%   $$E'_0=-\infty$$
%   or
%   $$E'_0>-\infty \text{ and \eqref{hp:sing} holds,}$$
%   then the following assertions hold:
% \begin{itemize}
% \item[a)] If $\mass<\mass_\rmc$,
%   % and $\urho_0^\perp$ has compact support, (questa ipotesi mi sembra non servire)
%   then, as $t\to+\infty$, $\urho_t$ weakly converges to $\urho_{\rm min}$ and
%   $u_t$ uniformly converges to $u_{\rm min}$ on compact sets of $\R$.
%   In particular
%   there is a regularization time $T$
%   such that $\urho_t^\perp\equiv0$ for all subsequent times $t>T$.
%  \item[ b)] If $\mass \ge \mass_\rmc$ and $V$ has a unique
%   minimum point $x_{\rm min}$,
%   then, as $t\to+\infty$, $\urho_t$ weakly converges to $\urho_{\rm min}$,
%   $u_t$ uniformly converges to $u_{\rm min}$ on compact sets of $\R\setminus \{x_{\rm min}\}$,
%   and
% %  and $\urho_0=u_0\Leb 1+\urho_0^\perp\in\CPlusMeasuresTwo\R$
% %  satisfies $\cI (\urho_0)<+\infty$(l'ipotesi precedente non dovrebbe servire per la disuguaglianza di regolarizzazione e la proprieta` di semigruppo)
% %  and $\urho_0^\perp$ has a compact support (questa ipotesi mi sembra non servire)
%   \begin{equation}
%     \label{eq:69}
%     \lim_{t\uparrow+\infty}{\rm dist}(x_{\rm min},\supp(\urho_t^\perp))=0,
%   \end{equation}
%   with the convention that ${\rm dist}(x_{\rm min},\emptyset)=0$.
%  \end{itemize}
  \end{theorem}

%%%%%%%%%%%%%%%%%%%%%%%%%%%%%%%%%%%%%%%%%%%%%%%%%%%%%%%%%%%%%%%%%%%%%
%%%%%%%%%%%%%%%%%%%%%%%%%%%%%%%%%%%%%%%%%%%%%%%%%%%%%%%%%%%%%%%%%%%%%

\section{Wasserstein distance and differential calculus}
\label{sec:Wass}

In this Section we recall the definition and the main properties of the Wasserstein distance and differential calculus in Wasserstein spaces (we refer the interested reader to \cite{V03}, \cite{V09}, \cite{ags} for
more details). Also, the subdifferential of the energy functional $\cF$ will
be characterized and discussed.

\subsection{Transport of measures, Wasserstein distance, and
  differential calculus.}
If $\urho\in \PlusMeasures{\R^d,\mass}$ and $\rr:\R^d\to \R^h$ is a Borel map, the push-forward of $\urho$ through $\rr$ is the measure
$\mu=\rr_\#\urho\in \PlusMeasures{\R^h,\mass}$ defined by
\begin{equation}
  \label{eq:33}
  \mu(A):=\urho(\rr^{-1}(A))\quad\text{for every Borel subset }A\subset \R^h.
\end{equation}
It can also be characterized by the change-of-variable formula
\begin{equation}
  \label{eq:34}
  \int_{\R^h}\varphi(y)\,\d\mu(y)=
  \int_{\R^d}\varphi(\rr(x))\,\d\urho(x),
\end{equation}
{for every bounded or
  nonnegative Borel function }$\varphi:\R^h\to\R.$\\
According to this definition, the marginals $\urho^i\in \PlusMeasures{\R,\mass}$, $i=1,2$, of $\rrho\in \PlusMeasures{\R\times \R,\mass}$ can be
defined by $\urho^i=(\pi^i)_\#\rrho$, where $\pi^i(x^1,x^2)=x^i$ is the projection on the $i$-th component in $\R\times \R$. In this case we say
that $\rrho$ is a coupling between $\urho^1,\urho^2$ and we denote by $\Gamma(\urho^1,\urho^2)$ the (weakly) closed convex subset of
$\PlusMeasures{\R\times\R,\mass}$ consisting of such couplings. We recall that a sequence of measures $\urho_n\in\PlusMeasures{\R^d,\mass}$
weakly converges to $\urho\in\PlusMeasures{\R^d,\mass}$ if
$\lim_{n\to+\infty}\int_{\R^d}\varphi(y)\,\d\urho_n(y)=\int_{\R^d}\varphi(y)\,\d\urho(y)$ for every continuous, bounded function $\varphi\in
C_{\rm b}(\R^d).$

For every couple of measures $\urho^1,\urho^2\in \PlusMeasuresTwo{\R,\mass}$ the $L^2$-Wasserstein distance is defined by
\begin{equation}
  \label{eq:35}
  W_2^2(\urho^1,\urho^2):=\min\Big\{\int_{\R\times\R}
    |x^1-x^2|^2\,\d\rrho(x^1,x^2):
    \rrho\in \Gamma(\urho^1,\urho^2)\Big\}.
\end{equation}
%\GGr{We say that $\urho_n\in\PlusMeasuresTwo{\R,\mass}$ converges in $\PlusMeasuresTwo{\R,\mass}$ to $\urho$ if
%$\lim_{n\to+\infty}W_2(\urho_n,\urho)=0$.}
The space $\PlusMeasuresTwo{\R,\mass}$ endowed with the distance $W_2$ is a complete separable metric
space and the topology induced by the Wasserstein distance is stronger than the narrow topology:
  in fact a sequence $\rho_n$ converges to $\rho$ in
$\PlusMeasuresTwo{\R,\mass}$ iff
\eqref{eq:14} holds (see e.g. \cite{V03}).

{There exists a unique optimal coupling $\rrho_{\rm opt}$ realizing
  the minimum in \eqref{eq:35}: it} admits a nice representation in terms of the cumulative distribution
functions $M_{\urho^i}$ of $\urho^1,\urho^2$ and of their
pseudo-inverses $Y_{\urho^i}$.

  Let us first recall their definitions in the case of $\sigma\in
  \PlusMeasures{\R,\mass}$
\begin{equation}
  \label{eq:36}
  M_{\sigma}(x):=\sigma\big((-\infty,x]\big)\quad x\in \R;\qquad
  Y_{\sigma}(w):=\inf\Big\{x\in \R: M_{\sigma}(x)\ge w\Big\},\quad
  w\in (0,\mass).
\end{equation}
Notice that $M_{\sigma}$ is a right-continuous and nondecreasing map
from $\R$ to $[0,\mass]$; if we denote by $\lambda_\mass=\Leb 1\restr{(0,\mass)}$ the restriction
of the Lebesgue measure to the interval $(0,\mass)$, it is possible to show that
\begin{equation}
  \label{eq:37}
  \big(Y_{\urho^i}\big)_\#\lambda_\mass=\urho^i,\quad
  \big(Y_{\urho^1},Y_{\urho^2}\big)_\#\lambda_\mass=\rrho_{\rm opt}
\end{equation}
  so that
  \begin{equation}
    \label{eq:72}
    W_2^2(\rho^1,\rho^2)=\int_0^\mass
    \big|Y_{\rho^1}(w)=Y_{\rho^2}(w)\big|^2\,\d w=
    \|Y_{\rho^1}-Y_{\rho^2}\|_{L^2(0,\mass)}^2.
  \end{equation}
  The map $\rho\mapsto Y_\rho$ provides an isometry between
  $\PlusMeasuresTwo{\R,\mass}$ and
  the cone of nondecreasing function in $L^2(0,\mass)$.

\paragraph{Displacement interpolation and displacement convexity.}
Let $\urho^0,\urho^1\in \PlusMeasuresTwo{\R,\mass}$.
Their \emph{displacement interpolation} is the path $\urho^\vartheta\in
\PlusMeasuresTwo{\R,\mass}$ with $\vartheta\in [0,1]$, defined by
\begin{equation}
  \label{eq:38}
  \urho^\vartheta:=\big((1-\vartheta)Y_{\urho^0}+\vartheta
  Y_{\urho^1}\big)_\#\lambda_\mass = \big((1-\vartheta)\pi^1+\vartheta
  \pi^2\big)_\# \rrho_{\rm opt}.
\end{equation}

The curve  $\vartheta\mapsto \rho^\vartheta$ is the unique (minimal, constant speed)
  geodesic connecting $\rho^0$ to $\rho^1$ in
  $\PlusMeasuresTwo{\R,\mass}$ and it corresponds to
  the segment connecting $Y_{\rho^0}$ to $Y_{\rho^1}$ in
  $L^2(0,\mass)$.

We say that a functional $\cG:\PlusMeasuresTwo{\R,\mass}\to
(-\infty,+\infty]$
is displacement $\lambda$-convex if for every $\urho^0,\urho^1$ in its proper domain we have
\begin{equation}
  \label{eq:39}
  \cG(\urho^\vartheta)\le
  (1-\vartheta)\cG(\urho^0)+\vartheta\cG(\urho^1)-\frac\lambda2\vartheta(1-\vartheta)
  W_2^2(\urho^0,\urho^1).
\end{equation}
In the one-dimensional case, the displacement convexity of the
internal functional $\cE$
is equivalent to the convexity of the energy density $E$
  and it coincides with convexity along generalized geodesics (see
  \cite[Definition 9.2.4]{ags}).

 \begin{proposition}[Displacement $\lambda$-convexity and lower semicontinuity of $\cF$]
 \label{prop:lscF}
 $\cF$ is lower semicontinuous with respect to the Wasserstein
 distance in $\PlusMeasuresTwo{\R,\mass}$
 and displacement $\lambda$-convex.
 Moreover $\cF$ satisfies the following coercivity property
 \begin{equation}
   \inf\Big\{\cF(\rho):\rho\in \PlusMeasuresTwo{\R,\mass},\quad
   \int_\R|x|^2\,\d\rho(x) \leq C\Big\} >-\infty\quad\text{for every }C>0.
 \end{equation}
 \end{proposition}
\begin{proof}
Since $E$ is convex and sublinear, by \cite{DT84} it follows that $\cE$ is lower semicontinuous with respect to the narrow convergence. In the
one-dimensional case the convexity of $E$ is equivalent to the
displacement convexity.
The functional $\urho\mapsto\int_\R V(x)\,\d\urho(x)$ is
displacement $\lambda$-convex if and only if $V$ is $\lambda$-convex;
  it is also lower semicontinuous with respect to convergence in
  $\PlusMeasuresTwo{\R,\mass}$ since $V$ is continuous and
  quadratically bounded from below.
\end{proof}

\begin{definition}[Subdifferential and slope]
  \label{def:subdifferential}
  Let $\cG:\PlusMeasuresTwo{\R,\mass}\to (-\infty,+\infty]$ be
  a displacement $\lambda$-convex and lower semcontinuous functional,
  let $\urho^0\in \PlusMeasuresTwo{\R,\mass}$ with
  $\cG(\urho^0)<+\infty$ and $\xxi\in L^2(\urho^0)$.
  We say that $\xxi$ belongs to the $W_2$-subdifferential
  of $\cG$ at the point $\urho^0$, and we write $\xxi\in \partial\cG(\urho^0)$, if for every $\urho^1\in
  \PlusMeasuresTwo{\R,\mass}$ {the optimal coupling
    $\rrho_{\rm opt}$ between $\urho^0$ and $\urho^1$ satisfies}
  \begin{equation}
    \label{eq:40}
    \cG(\urho^1)-\cG(\urho^0)\ge
    \int_{\R\times\R}\Big(\xxi(x)(y-x) + \frac\lambda2|y-x|^2\Big)\,\d\rrho_{\rm opt}(x,y).
  \end{equation}
  $\partial\cG(\urho^0)$ is a closed convex (and possibly empty)
  subset of $L^2(\urho^0)$.
  When $\partial\cG(\urho^0)$ is not empty
  we denote by
  $\partial^\circ\cG(\urho^0)\in L^2(\rho^0)$ its (unique) element
  of minimal $L^2(\urho^0)$-norm.

  The (metric) slope of $\cG$ is defined as
  \begin{equation}
    \label{eq:42}
    |\partial\cG|(\urho^0)=\limsup_{W_2(\urho,\urho^0)\to0}
    \frac{\big(\cG(\urho^0)-\cG(\urho)\big)^+}{W_2(\urho,\urho^0)}
    =\sup_{\rho\neq\rho^0}
    \Big(
    \frac{\big(\cG(\urho^0)-\cG(\urho)\big)^+}{W_2(\urho,\urho^0)}+
    \frac \lambda2 W_2(\rho,\rho^0)\Big)^+.
  \end{equation}
\end{definition}
For general displacement $\lambda$-convex functionals, one has
\begin{equation}
  \label{eq:76}
  |\partial\cG|(\urho)\le \|\partial^\circ\cG(\rho)\|_{L^2(\rho)}.
\end{equation}
When the functional $\cG$ satisfies the regularity condition
\begin{equation}
  \label{eq:73}
  |\partial\cG|(\rho^0)<+\infty\quad \Rightarrow\quad
  \rho^0\ll\Leb1,
\end{equation}
then the metric slope \eqref{eq:42} can be equivalently characterized by
\begin{equation}
  \label{eq:41}
  |\partial\cG|^2(\urho^0):=\min\Big\{\int_\R |\xxi|^2\,\d\urho^0:\xxi\in \partial\cG(\urho^0)\Big\},
\end{equation}
where $|\partial\cG|(\urho^0)=+\infty$ iff
$\partial\cG(\urho^0)$ is empty.
In this case $|\partial\cG|(\urho^0)=\|\partial^\circ\cG(\urho^0)\|_{L^2(\urho^0)}$.

\subsection{Slope and Fisher dissipation in the super-linear case.}
Let us consider the perturbed family of energy densities
$E^\eps(r):=E(r)+\eps r\log r$
associated to the energy functionals
\begin{equation}
  \label{eq:400}
  \cF^\eps(\urho):=\int_\R E^\eps(u(x))\,\d x+\int_\R
  V^\eps(x)\,\d\urho(x)\quad\text{if } \urho=u\,\Leb 1;\quad
  \cF^\eps(\urho)=+\infty\quad\text{if }\urho\not\ll\Leb 1.
\end{equation}
Notice that $(rE^\eps)''(r)=\beta'(r)+\eps=(\beta^\eps)'(r)$, where $\beta^\eps$ is defined in \eqref{eq:16}.
%\begin{equation}
%  \label{eq:45}
%  \beta^\eps(r):=\beta(r)+\eps r.
%\end{equation}
Since $E^\eps$ has a super-linear growth, the slope $|\partial
\cF^\eps|$ can be explicitly characterized
\cite[Theorem 10.4.13]{ags} and it
coincides with the square root of the associated Fisher-dissipation
\begin{equation}
  \label{eq:46}
  \cI^\eps(\urho):=\int_\R\Big|\frac{\partial_x\beta^\eps(u)}u+(V^\eps)'\Big|^2u\,\d x \qquad\text{if }\urho=u\Leb1,\quad
  u\in W^{1,1}_{\rm loc}(\R).
\end{equation}
As usual $\cI^\eps(\urho)=+\infty$ if $u\not\in W^{1,1}_{\rm loc}(\R)$ or even $\urho\not\ll\Leb 1$. Thus we have
\begin{equation}
  \label{eq:47}
  |\partial \cF^\eps|^2(\urho)=\cI^\eps(\urho)
\end{equation}
and the minimal subdifferential $\xxi^\eps=\partial^\circ\cF^\eps(\urho)\in L^2(\urho)$ is characterized as
\begin{equation}
  \label{eq:48}
  \xxi^\eps\urho=\partial_x\beta^\eps(u)\Leb1+\urho\,(V^\eps)'\quad\text{if }\urho=u\Leb1\in D(\cI^\eps).
\end{equation}
The following compactness and lower semicontinuity property will play a crucial role in the sequel.

\begin{theorem}
  \label{thm:lsc-dissipation}
  If $\urho^\eps=u^\eps\,\Leb 1\in D(\cI_{\eps})$, $\eps>0$, with $u^\eps(x)>0$  for all $ x\in\R$,
  is a family of measures satisfying
  \begin{equation}\label{wconv}
    {\urho^\eps\weakto \urho\quad\text{weakly in
      $\PlusMeasures{\R,\mass}$ as $\eps\downarrow0,$}}\quad
    \limsup_{\eps\downarrow0} \cI^{\eps}(\urho^\eps)<+\infty,
  \end{equation}
  % \begin{equation}\label{upbound}
  % \end{equation}
  then we have
  \begin{equation*}
    \urho=u\,\Leb 1+\urho^\perp\in D(\cI)\subset \CPlusMeasuresTwo{\R,\mass},
  \end{equation*}
  \begin{equation}
    \label{eq:490}
    \cI(\urho)\le \liminf_{\eps\downarrow0}\cI^{\eps}(\urho^\eps),
  \end{equation}
  \begin{equation}\label{eq:219}
  u^\eps \text{ converges to } u \text{ uniformly on compact sets of } \Dom{u}.
  \end{equation}
  % \begin{equation}\label{eq:219}
%\fcolorbox{red}{green}{$ \beta_{\eps}(u^\eps) \text{ converges to } \beta(u) \text{ locally uniformly in } \R$},
%  \end{displaymath}
%  \begin{equation}\label{eq:220}
%\fcolorbox{red}{green}{$\beta(u_{\eps}) \text{ converges to }\beta(u) \text{ locally uniformly in }\R$}.
%  \end{equation}
  Moreover, if $\xxi^\eps = \partial^o\cF^\eps(\urho^\eps)$
  as in \eqref{eq:48}, we have
  \begin{equation}
    \label{eq:49}
    \xxi^\eps\urho^\eps \weakto \xxi\urho=\partial_x\beta(u)\Leb1+V'\urho,\quad\text{in
      the duality with }C^0_{\rm b}(\R).
  \end{equation}
  Finally, if $f:[0,+\infty)\to \R$ is a continuous function
  such that $\displaystyle\lim_{r\uparrow+\infty}\frac{f(r)}{r}=f_\infty \in\R$, then
  \begin{equation}
    \label{eq:65}
    f(u^\eps)\Leb 1\weakto f(u)\Leb1+f_\infty\urho^\perp\quad\text{in
      the duality with }C^0_{\rm c}(\R).
  \end{equation}
\end{theorem}

\begin{proof}
Since $\displaystyle \cI^\eps (\urho^\eps) = \int_\R |\xxi^\eps|^2 \d \urho^\eps$, by \eqref{wconv} (see \cite[Theorem 5.4.4]{ags}) there exists $\xxi\in
L^2(\urho)$ such that
\begin{equation}\label{eq:401}
    \xxi^\eps\urho^\eps \weakto \xxi\urho,\quad\text{in
      the duality with }C^0_{\rm b}(\R),
\end{equation}
and
\begin{equation*}
   \int_\R |\xxi|^2 \d \urho\le  \liminf_{\eps \downarrow 0} \int_\R |\xxi^\eps|^2 \d \urho^\eps.
\end{equation*}
From \eqref{wconv} and \eqref{eq:85} it follows that
\begin{equation}\label{eq:402}
(V^\eps)'u^\eps \Leb1 \weakto V'\urho \quad \text{ in the duality with }C^0_{\rm c}(\R).
\end{equation}
Since by \eqref{eq:48}
\begin{equation}\label{eq:80}
\de_x\beta^\eps(u^\eps)\Leb1 = \xxi^\eps u^\eps\Leb1 - (V^\eps)'u^\eps\Leb1,
\end{equation}
\eqref{eq:401} and \eqref{eq:402} imply that
$$\de_x\beta^\eps(u^\eps)\Leb1 \weakto  \xxi \urho - V'\urho \quad\text{ in the duality with }C^0_{\rm c}(\R).$$
Let us now prove that $\urho\in \CPlusMeasuresTwo{\R,\mass}$, $\beta(u)\in W^{1,1}_\loc(\R)$ and $\de_x\beta(u)= \xxi \urho - V'\urho$.
\\
We introduce the functions
$$
G(r)=\int_0^r \frac{\beta'(s)}{\sqrt s}\,\d s, \qquad G^\eps(r)=G(r)+2\eps\sqrt r.
$$
Since $u^\eps \in W^{1,1}_\loc(\R)$, $u^\eps(x)>0$ and $G$ is locally Lipschitz in $(0,+\infty)$, we have
\begin{equation}\label{sss}
\de_xG^\eps(u^\eps) = \frac{\de_x(\beta^\eps(u^{\eps}))}{\sqrt{u^\eps}}.
\end{equation}
Let $I=(a,b)$ be an arbitrary bounded interval of $\R$. Since $\beta'(0^+) < +\infty$ we have that $G^\eps (r)\le M\sqrt r$, for some $M>0$.
Therefore
\begin{equation}\label{bound-11}
\sup_\eps \int_I\left|G^\eps(u^\eps)\right|^2 \,\d x < +\infty.
\end{equation}
By \eqref{eq:80} and \eqref{wconv} we have
\begin{equation} \label{bound-l2}
%\sup_\eps
  \int_I\left|\frac{\de_x(\beta^\eps(u^{\eps}))}{\sqrt{u^\eps}}\right|^2 \d x =
%\sup_\eps
  \int_I |\xxi^\eps - (V^\eps)'|^2u^\eps\d x \le
 2
 %\sup_\eps
 \int_I |\xxi^\eps|^2u^\eps\d x + 2
 %\sup_\eps
 \int_I |(V^\eps)'|^2u^\eps\d x
 %< +\infty.
\end{equation}
so that
\begin{equation}
  \label{eq:84bis}
  \sup_{\eps>0}
  \int_I\left|\frac{\de_x(\beta^\eps(u^{\eps}))}{\sqrt{u^\eps}}\right|^2 \d x<+\infty.
\end{equation}
By \eqref{bound-11}, \eqref{eq:84bis} and \eqref{sss}, we infer that the family $\{G^\eps(u^\eps)\}_{\eps>0}$ is bounded in $H^1_\loc(\R)$. Thus,
for every sequence $\eps_j\to 0$ we can extract a sub-sequence, still denoted by $\{\eps_j\}$, such that $G_{\eps_j}(u_{\eps_j})$ converges weakly in $H^1_\loc(\R)$, and uniformly on the compact sets of $\R$,
to some continuous function $g \in H^1_\loc(\R)$.
Since
\begin{equation}\label{bvest}
  \sup_\eps\int_I |\de_x(\beta^\eps(u^{\eps}))|\,\d x=\sup_\eps\int_I | \xxi^\eps-(V^\eps)'|u^{\eps}\,\d x\le \sup_\eps\sqrt{\mass}\bigg(\int_I | \xxi^\eps-(V^\eps)'|^2u^{\eps}\,\d x\bigg)^\frac 1 2 <+\infty,
\end{equation}
and $\{\beta^\eps(u^\eps)\}_{\eps>0}$ is bounded in $L^1(\R)$,
the family  $\{\beta^\eps(u^\eps)\}_{\eps>0}$ is bounded in $L^\infty(I)$.
Therefore the family $\{\eps u^\eps = \beta^\eps(u^\eps)-\beta(u^\eps)\}_{\eps>0}$ is bounded in $L^\infty(I)$. Since $0\leq
G^\eps(u^\eps)-G(u^\eps)=2\sqrt{\eps} \sqrt{\eps u^\eps}$, we conclude that $G(u_{\eps_j})$ converges uniformly on the
compact sets of $\R$ to $g$, as $j\to +\infty$. The inequality
$$
 0\leq G \leq G_\infty=\int_0^{+\infty} \frac{\beta'(s)}{\sqrt s}\d s,
$$
together with the previous observations, gives $0\le g \le G_\infty$. Since $G$ is increasing and $G_\infty < +\infty$, we can define
the function
$$u(x):=\begin{cases}G^{-1}(g(x)) & \text{if } g(x)<G_\infty,\\
                     +\infty & \text{if } g(x)=G_\infty\end{cases}$$
which turns out to be continuous on the open set $\Dom u:=\{ x\in\R: g(x)<G_\infty\}$. Since $G(u_{\eps_j}) \to g$ uniformly on the compact
sets of $\R$, we have that $u_{\eps_j} = G^{-1}(G(u_{\eps_j}))\to u$ on the compact sets of $\Dom u$ and $u_{\eps_j}(x)\to+\infty$ for every $x\in\R\setminus
\Dom u$. By Fatou's Lemma we obtain that $u\in L^1(\R)$ and $\Leb{1}(\R\setminus \Dom u)=0$.
For every $\psi\in C^0_{\rm c}(\Dom u)$, using (\ref{wconv}) we have
$$
\lim_{j\to+\infty} \int_\R\psi(x)\,\d\urho_{\eps_j}=  \int_\R\psi(x)\,\d\urho =\int_\R\psi(x)u(x)\,\d x.
$$
Thus
\begin{equation} \label{eq:decomp}
\urho_{|\Dom u}=u\Leb1  \quad \text{ and }\quad  \urho_{|\R\setminus \Dom u}= \urho^\perp.
\end{equation}
This shows that $\urho\in \CPlusMeasuresTwo{\R,\mass}$. Moreover, we deduce that the whole family $u^\eps$ converges to $u$ uniformly on compact sets of $\Dom u$, as $\eps \downarrow 0$.

For any  bounded interval $I=(a,b)$, we have proved that $\{\beta^\eps(u^\eps)\}_{\eps>0}$ is bounded in $W^{1,1}(I)$.
Then, by ${\rm BV}$ compactness (see e.g. \cite{afp}) there exists $h\in {\rm BV}_{\rm loc}(\R)$ such that, up to subsequences as before, $\beta^\eps(u^\eps)\to h$ in $L^1_{\rm loc}(\R)$ and $\Leb{1}$-a.e. and $\partial_x\beta^\eps(u^\eps)\Leb{1}  \weakto \partial_x h$ in duality with $C_{\rmc}^0(\R)$.
Since $ 0 \le \beta^\eps (u^\eps) - \beta(u^\eps) = \eps u^\eps$ and $\eps u^\eps(x) \to 0$ pointwise in $\Dom u$, we have that $\beta(u^\eps) \to h$, $\Leb{1}$-a.e.
On the other hand, by the continuity of $\beta$, $\beta(u^\eps) \to \beta(u)$ $\Leb{1}$-a.e. Hence $h=\beta(u)$.
Moreover, by using \eqref{eq:80}, it is easy to see that $\partial_x\beta(u)\Leb1= \xxi \urho - V'\urho$.
The last identity and \eqref{eq:decomp} yield $\beta(u)\in {\rm BV}_{\rm loc}(\R) \cap W^{1,1}_\loc(\Dom u)$.

Finally, we prove that $\beta(u)\in W^{1,1}_\loc(\R)$ and $\de_x(\beta(u))=\de_x(\beta(u))_{|\Dom u}$.
Since $\Dom u$ is open, we can write
$$
\Dom u=\bigcup_{n\in\N} (a_n,b_n)
$$
where the intervals are pairwise disjoint;
recalling that $\beta(u(a_n))=\beta(u(b_n))=\betainfty$, we have for every $\zeta\in C_{\rm c}^\infty(\R)$
\begin{align*}
\int_\R \zeta'\beta(u)\,\d x&=\int_{\Dom u}
\zeta'\,\beta(u)\,\dx+{\int_{\R\setminus \Dom
    u}\zeta'\,\beta(u)\,\dx}=\sum_n\int_{a_n}^{b_n}
\zeta'\beta(u)\,\dx
{+\betainfty\int_{\R\setminus\Dom u}\zeta'\,\dx}\\
&=\sum_n\bigg(-\int_{a_n}^{b_n}
\zeta\,\de_x(\beta(u))\,\dx+\big(\zeta({b_n})-\zeta({a_n})\big)\betainfty\bigg)
{+\betainfty\int_{\R\setminus\Dom u}\zeta'\,\dx}\\
&=-\int_{\Dom u} \zeta\,\de_x(\beta(u))\,\dx+\sum_n\betainfty
\int_{a_n}^{b_n} \zeta'\,\dx
{+\betainfty\int_{\R\setminus\Dom u}\zeta'\,\dx}\\
&=-\int_{\Dom u} \zeta\,\de_x(\beta(u))\,\dx+\betainfty \int_\R \zeta'\,\dx
=-\int_{\Dom u} \zeta\,\de_x(\beta(u))\,\dx .
\end{align*}
%\GGr{and we conclude.}

We eventually prove \eqref{eq:65}.
  By possibly substituting $f(r)$ with $f(r)-f_\infty r$ it is not
  restrictive to assume $f_\infty=0$, i.e.
  \begin{equation}
    \label{eq:74}
    \lim_{r\to+\infty}\frac{f(r)}r=0\quad\text{or, equivalently,}\quad
    \forall\,\eta>0\ \exists\, M_\eta:\quad
    |f(r)|\le M_\eta+\eta r\quad\text{for every }r\ge0.
  \end{equation}
  Property \eqref{eq:74} easily shows that the family
  $\{f(u^\eps)\}_{\eps>0}$ is equi-integrable in $\R$:
  for every $\delta>0$ and choosing $\eta:= \delta/2\mass $,
  every Borel set $A$ with
  measure $\Leb1(A)\le \delta/2M_\eta$ satisfies
  \begin{equation}
    \label{eq:75}
    \int_A |f(u^\eps(x))|\,\d x\le \int_A \Big(M_\eta +\eta
    u_\eps(x)\Big)\,\d x\le M_\eta\, \Leb 1(A)+\eta\, \mass\le
    \delta\quad
    \text{for every }\eps>0.
  \end{equation}
  The previous equi-integrability estimate and the tightness of
  $\rho^\eps$ given by \eqref{wconv} show that the family
  $f(u^\eps)$ is weakly
  compact in $L^1(\R)$. On the
  other hand, $f(u^\eps)\to f(u)$ locally uniformly in $\Dom u$. Since
  $\Leb1(\R\setminus\Dom u)=0$ it follows that $f(u)$ is also the weak
  limit of $f(u^\eps)$ in $L^1(\R)$.

\end{proof}
By a similar and even simpler argument it is possible to prove the following lower semi
continuity result for the Fisher dissipation
  $\cI$ with respect to weak convergence. Lower semicontinuity with
  respect to Wasserstein convergence will follow
  by \eqref{th:charsubdiff} and the representation
  \eqref{eq:42} of the metric slope for a displacement
  $\lambda$-convex functional \cite[Corollary 2.4.10]{ags}.

\begin{theorem}[Lower semi continuity of $\cI$]
\label{thm:lsc-dissipation2}
  If $\urho_n=u_n\,\Leb 1 +\urho^\perp_n \in D(\cI)$
  is a sequence of measures weakly convergent to a measure $\urho$ and satisfying
  \begin{equation}\label{upbound2}
    \limsup_{n\to+\infty} \cI(\urho_n)<+\infty,
  \end{equation}
  then we have
  \begin{equation}\label{eq:4902}
    \urho=u\,\Leb 1+\urho^\perp\in D(\cI)\subset \CPlusMeasuresTwo{\R,\mass},
    \qquad
    \cI(\urho)\le \liminf_{n\to+\infty}\cI(\urho_n).
  \end{equation}
  Moreover
  \begin{equation}\label{eq:2190}
  u_{n} \text{ converges to } u \text{ uniformly on compact sets of } \Dom{u}.
  \end{equation}
\end{theorem}
% \begin{proof}
% Since $\cI(\urho_n) = \int_\R |\xxi_n|^2 \d \urho_n$, where $\xxi_n\urho_n=\de_x\beta(u_n)\Leb1 + V'\urho_n$,
% by \eqref{upbound2} (see \cite[Theorem 5.4.4]{ags}) there exists $\xxi\in
% L^2(\urho)$ such that
% \begin{equation}\label{eq:401}
%     \xxi_n\urho_n \weakto \xxi\urho,\quad\text{in
%       the duality with }C^0_{\rm b}(\R),
% \end{equation}
% and
% \begin{equation*}
%    \int_\R |\xxi|^2 \d \urho\le  \liminf_{n\to+\infty} \int_\R |\xxi_n|^2 \d \urho_n.
% \end{equation*}
% From the narrow convergence of $\urho_n$ to $\rho$ it follows that
% \begin{equation}\label{eq:4020}
% V'\urho_n \weakto V'\urho \quad \text{ in the duality with }C^0_{\rm c}(\R).
% \end{equation}
% Since
% \begin{equation}%\label{eq:80}
% \de_x\beta(u_n)\Leb1 = \xxi_n \urho_n - V'\urho_n,
% \end{equation}
% \eqref{eq:401} and \eqref{eq:4020} imply that
% $$\de_x\beta(u_n)\Leb1 \weakto  \xxi \urho - V'\urho \quad\text{ in the duality with }C^0_{\rm c}(\R).$$
% Since, for any bounded interval $I$, $\displaystyle\int_I |\de_x(\beta(u_n))|\,\d x<+\infty$
% we obtain that there exist $h\in BV_\loc(\R)$ such that
% $\beta(u_n)$ converges to $h$ in $L^1_\loc(\R)$ and $\de_x\beta(u_n)\Leb1$ converges to $\de_x h$ in the duality with $C_c^0(\R)$.
% By the same argument of Theorem \ref{thm:lsc-dissipation},
% we can prove that $h=\beta(u)$ and \eqref{eq:2190} holds.
% \end{proof}

\subsection{Characterization of the Wasserstein subdifferential of $\cF$}

\begin{theorem}[Characterization of $\de\cF$]\label{th:charsubdiff}
Let $\urho = u\Leb{1}+\urho^\perp\in \PlusMeasuresTwo{\R,\mass}$ with $\cF(\urho)<+\infty$ and
$\xxi\in L^2(\urho)$.\\
$\xxi=\de^o\cF(\urho)$ (and, in particular, $\partial\cF(\urho)$ is not empty) if and only if
\begin{equation}
  \label{eq:50}
  \urho\in \CPlusMeasuresTwo{\R,\mass}, \quad \cI(\urho)<+\infty,
  \quad
  \xxi\urho=\de_x\beta(u)\Leb{1}+V'\urho.
\end{equation}
In this case
\begin{equation}
  \label{eq:44}
  |\partial\cF|^2(\urho)=\int_\R |\xxi|^2\,\d\urho=\cI(\urho).
\end{equation}
\end{theorem}
\begin{proof}
Let us first suppose that $\urho,\xxi$ satisfy \eqref{eq:50} and let
us prove that $\xxi\in \partial\cF(\urho)$, i.e.\ \eqref{eq:40} holds with
$\urho^0:=\urho$;
  in particular, recalling \eqref{eq:76}, this also shows that
  \begin{equation}
    \label{eq:77}
    |\partial\cF|^2(\rho)\le \int_\R |\xxi|^2\,\d\rho=\cI(\rho)<+\infty.
  \end{equation}

It is not restrictive to assume $\lambda=0$. By a standard regularization and stability of the optimal couplings with respect to
weak convergence, we can also suppose that $\urho^1=u^1\,\Leb1$ with $u^1\in C^1(\R)$ supported in the bounded interval $[a,b]$ with $u^1(x)>0$
for every $x\in (a,b)$. In this case $M_{\urho^1}\in C^2(\R)$, the monotone rearrangement map $Y_{\urho^1}\in C^0([0,\mass])$ satisfies
$Y_{\urho^1}(0)=a,\ Y_{\urho^1}(\mass)=b$ and its restriction to $(0,\mass)$ is of class $C^2$. We set
\begin{equation}
  \label{eq:51}
  \left\{
    \begin{aligned}
      \rr(x):=&Y_{\urho^1}(M_\urho(x)),\ \rr^\vartheta(x):=(1-\vartheta)x+\vartheta\rr(x)\\
      \ss(y):=&Y_\urho(M_{\urho^1}(y)),\ \ss^{\vartheta}(y):=\vartheta y+(1-\vartheta)\ss(y)
    \end{aligned}
\right.
\quad\text{for
    every }x,y\in \R,\ \vartheta\in [0,1],
\end{equation}
and we observe that $\rr^\vartheta\restr{\Dom u}\text{ is $C^1$.}$
We introduce the sets
\begin{equation}
  \label{eq:56}
  \sfD:=\Dom u,\quad \sfD_>:=\{x\in \Dom u:u(x)>0\},\quad
  \stD:=\R\setminus \sfD,\quad
  \sfG:=\rr(\sfD)=\rr(\sfD_>),\quad
  \tilde\sfG:=(a,b)\setminus \sfG,
\end{equation}
and we have
\begin{gather}
  \label{eq:52}
  (\rrho_{\rm opt})\restr{\sfD\times\R}=(\ii\times\rr^1)_\# (u\Leb
  1)=
  (\ss^0\times \ii)_\#(u^1\Leb1\restr\sfG)
  ,\quad
  (\rrho_{\rm opt})\restr{\stD\times\R}=
  (\ss^0\times \ii)_\#(u^1\Leb1\restr{\tilde\sfG})
 \\
 \urho^\vartheta\restr{\srr^\vartheta(\sfD)}=\rr^\vartheta_\# (u\Leb 1),\quad
 \urho^\vartheta\restr{\R\setminus\srr^\vartheta(\sfD)}=\ss^{\vartheta}_\# (u^1\Leb 1\restr{\tilde\sfG}),\quad
 \\
  \label{eq:53}
  u^\vartheta(\rr^\vartheta(x))(\rr^\vartheta)'(x)=u(x),\quad
  u^\vartheta(\ss^\vartheta(y))(\ss^\vartheta)'(y)=u^1(y)\quad\text{for every }x\in \sfD,
  \ y\in (a,b).
\end{gather}
Since $(\ss^0)'(y)=0$ for every $y\in \tilde G$
\begin{equation}
  \label{eq:55}
  \cE(\urho^\vartheta)=\int_{\sfD_>}E\Big(\frac{u(x)}{(1-\vartheta)+\vartheta\rr'(x)}\Big)(1-\vartheta+\vartheta\rr'(x))\,\d
  x+\int_{\tilde \sfG}E\Big(\frac{u^1(y)}{\vartheta}\Big)\vartheta\,\d y.
\end{equation}
Therefore, owing to the convexity of the maps $\vartheta\mapsto
\cE(\urho^\vartheta)$
{and $s\mapsto sE(\alpha/s)$ for every $\alpha\ge0$,}
\begin{align*}
  +\infty>\cE(\urho^1)-\cE(\urho)\ge\lim_{\vartheta\downarrow0}\vartheta^{-1}\Big(\cE(\urho^\vartheta)-\cE(\urho)\Big)=
  -\int_{\sfD}\beta(u)(\rr'-1)\,\d x-\betainfty\Leb 1(\tilde \sfG).
\end{align*}
Let us now choose two sequences $z^-_{k}\to-\infty$, $z^+_k\to +\infty$ in $\sfD$, let $(a_k^-,b_k^-) $. Let $(a_k^+,b_k^+)$ be  the connected
component of $\sfD$ containing $z_k^-$ and $z_k^+$ respectively, and let $I_k^n:=(a_k^n,b_k^n)$, $n\in \Lambda_k\subset \N$ be the (at most
countable) connected components of $\sfD\cap (b_k^-,a_k^+)$. We consider a continuous function $\psi_k:\R\to [0,1]$ satisfying
\begin{equation}
  \label{eq:57}
  \psi_k(x)=0\text{ in }\R\setminus [z_k^-,z_k^+],
  \quad \psi_k(x)\equiv 1\text{ if }x\in [\tfrac
  12(z_k^-+b_k^-),\tfrac 12 (z_k^++a_k^+)],\quad
  \psi_k\restr{[z_k^-,z_k^+]}\text{ is concave.}
\end{equation}
For sufficiently big $k$ we have $\psi_k\equiv 1$ on $(a,b)$. Then
\begin{equation}
  \label{eq:60}
  \begin{aligned}
  \beta(u(x))(\rr(x)-x)\psi_k'(x) &\ge
  \beta(u(x))\big(\psi_k(\rr(x))-\psi_k(x)) \\ &\ge
  \beta(u(x))(1-\psi_k(x))\ge0\quad\text{for every }x\in [z_k^-,z_k^+];
    \end{aligned}
\end{equation}
\begin{equation}
  \label{eq:61}
  -\betainfty\Leb 1(\tilde \sfG)=
  \lim_{k\to\infty}\Leb 1(\tilde\sfG\cap(\rr(b_k^-),\rr(a_k^+))).
\end{equation}
Moreover
\begin{equation}
  \label{eq:58}
  -\int_{\sfD}\beta(u)(\rr'-1)\,\d x=\lim_{k\uparrow+\infty}
  -\int_{\sfD}\beta(u)(\rr'-1)\,\psi_k(x)\,\d x
\end{equation}
and
\begin{align*}
  -\int_{\sfD}&\beta(u)(\rr'-1)\,\psi_k(x)\,\d x\ge
  \int_{a_k^+}^{z_k^+}\partial_x \beta(u)\,(\rr(x)-x)\psi_k(x)\,\d x
  +
  \int_{z_k^-}^{b_k^-}\partial_x \beta(u)\,(\rr(x)-x)\psi_k(x)\,\d x
  \\&+
  \sum_{n\in \Lambda_k} \int_{a_k^n}^{b_k^n}\partial_x
  \beta(u)\,(\rr(x)-x)\,\d x
  \\&+\betainfty\Big[(\rr(a_k^+)-a_k^+)-(\rr(b_k^-)-b_k^-)-
  \sum_{n\in \Lambda_k}
  \big(\rr(b_k^n)-\rr(a_k^n)-(b_k^n-a_k^n)\big)\Big]
  \\&=
  \int_\R \partial_x\beta(u)(\rr(x)-x)\psi_k(x)\,\d x+
  \betainfty\Leb 1(\tilde\sfG\cap(\rr(b_k^-),\rr(a_k^+))),
\end{align*}
  where we used the fact that $\Leb1\big((b_k^-,a_k^+)\setminus \sfD\big)=0$.

Combining all these estimates we get
\begin{equation}
  \label{eq:62}
  +\infty>\cE(\urho^1)-\cE(\urho)\ge\limsup_{k\uparrow+\infty}\int_\R \partial_x\beta(u)(\rr(x)-x)\psi_k(x)\,\d x.
\end{equation}
On the other hand
\begin{align*}
  +\infty&>\int_\R V(y)\,\d\urho^1(y)-\int_\R V(x)\,\d\urho(x)=
  \int_{\R\times\R}\Big(V(y)-V(x)\Big)\,\d\rrho_{\rm opt}(x,y)\\&\ge
  \int_{\R\times\R} V'(x)(y-x)\,\d\rrho_{\rm opt}(x,y)\ge
  \limsup_{k\uparrow+\infty}  \int_{\R\times\R} V'(x)(y-x)\psi_k(x)\,\d\rrho_{\rm opt}(x,y)
\end{align*}
% with quadratic growth, i.e.,
% satisfying $|\zeta(x,y)|\le A+B(x^2+y^2),$ $\forall x,y\in \R$ and for
% some positive constants $A,B$.
Summing up the two contributions we have
\begin{displaymath}
  \cF(\urho^1)-\cF(\urho)\ge \limsup_{k\uparrow+\infty}
  \int_{\R\times\R} \xxi(x)(y-x)\psi_k(x)\,\d\rrho_{\rm opt}(x,y)=
  \int_{\R\times\R} \xxi(x)(y-x)\,\d\rrho_{\rm opt}(x,y).
\end{displaymath}

  Let us now show that if $|\partial\cF|(\rho)<+\infty$ then there
  exists $\xxi\in L^2(\rho)$ satisfying \eqref{eq:50} (thus in
  particular $\xxi\in \partial\cF(\rho)$) with
  \begin{equation}
    \label{eq:78}
    \cI(\rho)=\int_\R |\xxi|^2\,\d\rho\le |\partial\cF|^2(\rho);
  \end{equation}
  recalling \eqref{eq:77}, this shows that $\xxi=\partial^\circ\cF(\rho)$.

We apply the forthcoming Lemma \ref{le:Gamma-convergence} and the general
approximation result \cite[Lemma 10.3.16]{ags} to find a family $\urho_{\eps}$ converging to $\urho$ in $\PlusMeasuresTwo{\R,\mass}$ and
$\xxi^\eps\in\de \cF^{\eps}(\urho_{\eps})$
%\GGr{$\xxi\in L^2(\urho)$}
such that
  \begin{equation}
    % \label{convergenza-xi}
    \lim_{\eps\downarrow0} |\partial\cF^\eps|^2(\rho^\eps)=
    \lim_{\eps\downarrow0} \cI^\eps(\rho^\eps)
    = \lim_{\eps\downarrow0} \int_\R |\xxi^\eps|^2\,\d\rho^\eps=
    |\partial\cF|^2(\rho).
  \end{equation}
  Theorem
  \ref{thm:lsc-dissipation} then yields \eqref{eq:78} and \eqref{eq:50}.
\end{proof}

\subsection{$\Gamma$-convergence of $\cF^\eps$ to $\cF$}
The following lemma shows that the family of functionals
$\cF^\eps$ converges to $\cF$ in a kind of $\Gamma$ convergence way
(with different convergence in the $\liminf$ and the $\limsup$ inequalities).
\begin{lemma}
  \label{le:Gamma-convergence}
As $\eps\downarrow0$ the family of functionals $\cF^\eps$ converge to $\cF$
according to the following two properties:
\begin{itemize}
\item[(i)] For every family $\{\urho^\eps\}\subset\PlusMeasuresTwo{\R,\mass}$
such that $\urho^\eps\weakto \urho$, as $\eps\downarrow0$, in duality with $C^0_{\rm b}(\R)$, and
\begin{equation}\label{boundsecmom}
M_2:=\limsup\limits_{\eps\downarrow0}\Mom{2}{\urho^\eps} <+\infty,
\end{equation}
one has
$$
\liminf_{\eps\downarrow 0}\cF^\eps(\urho^\eps)\ge \cF(\urho).
$$
\item[(ii)] For every $\urho \in \PlusMeasuresTwo{\R,\mass}$ there exists a family of measures
$\{\urho^\eps\}\subset \PlusMeasuresTwo{\R,\mass}$ such that $W_2(\urho^\eps,\urho)\to 0$ as $\eps\downarrow0$ and
$$
\limsup_{\eps\downarrow 0}\cF^\eps(\urho^\eps)\le \cF(\urho).
$$
\end{itemize}
\end{lemma}
\begin{proof}
(i)
  The ``liminf'' inequality for the potential energy $\displaystyle\cV^\eps(\rho):=
  \int_\R V^\eps\,\d\rho$ under weak convergence and \eqref{boundsecmom}
  follows from \eqref{eq:86} and \eqref{eq:85}, since for every
  $\delta>0$ there exist $R>\delta^{-1}$ and
  $\eps_0>0$ such that
  \begin{displaymath}
    V^\eps(x)\ge -\delta |x|^2\quad\text{for every }x\in \R\setminus
    [-R,R],\quad
    V^\eps(x)\ge V(x)-\delta \quad\text{for every }x\in [-2R,2R],\ 0<\eps<\eps_0;
  \end{displaymath}
  for every $0<\eps<\eps_0$ and every smooth function
  \begin{equation}
    \text{$\nchi:\R\to[0,1]$ with $\nchi(x)=1$
  if $|x|\le 1$ and $\nchi(x)=0$ if $|x|\ge2$}
\label{eq:88}
\end{equation}
we have
  \begin{displaymath}
    \cV^\eps(\rho^\eps)= \int_\R
    V^\eps(x)\nchi(x/R)\,\d\rho^\eps+
    \int_\R
    V^\eps(x)(1-\nchi(x/R))\,\d\rho^\eps\ge
    \int_\R
    \Big(V(x)-\delta\Big)\nchi(x/R)\,\d\rho^\eps
    -\delta\int_\R |x|^2\,\d\rho^\eps
 \end{displaymath}
 so that
 \begin{displaymath}
   \liminf_{\eps\to0}\cV^\eps(\rho^\eps)\ge
   \int_\R \nchi(x/R)V(x)\,\d\rho(x)-\delta(\mass+M_2).
 \end{displaymath}
 Since $R\ge\delta^{-1}$ and the previous inequality is valid for
 arbitrary $\delta>0$, passing to the limit as $\delta\to0$ we obtain
 \begin{equation}
   \label{eq:87}
   \liminf_{\eps\to0}\cV^\eps(\rho^\eps)\ge \cV(\rho).
 \end{equation}
Let us now prove the ``liminf'' inequality for $\cE^\eps$: recalling
the usual decomposition
$\urho^\eps=u^\eps\Leb{1} + (\urho^\eps)^\perp$, thanks to the definition of $\cE^\eps$ we get
\begin{equation*}
\cE^\eps(\urho^\eps)= \cE (\urho^\eps) + \eps\int_\R u^\eps \log u^\eps \d x \ge \cE(\urho^\eps)+\eps\int_{\{0< u^\eps< 1\}} u^\eps \log u^\eps \d x.
\end{equation*}
By Cauchy-Schwarz inequality and \eqref{boundsecmom} we obtain
\begin{align}      \label{stima-ulogu}
\limsup_{\eps\downarrow0}\bigg| \int_{\{0< u^\eps< 1\}} u^\eps \log u^\eps \,\dx\bigg|
& \le \limsup_{\eps\downarrow0}\bigg(\int_\R (1+|x|)^2 u^\eps \,\dx\bigg)^\frac 1 2
\bigg(\int_{\{0< u^\eps< 1\}} \frac{u^\eps \log^2 u^\eps }{(1+|x|)^2} \,\dx\bigg)^\frac 1 2 <+\infty.
 \end{align}
Hence
$$
\liminf_{\eps\downarrow0}\cE^\eps(\urho^\eps) \ge \liminf_{\eps\downarrow0}\cE(\urho^\eps),
$$
and (i) follows by the lower semicontinuity of $\cE$ with respect to the weak convergence.

(ii) Let $\urho=u\Leb1+\urho^\perp\in\PlusMeasuresTwo{\R,\mass}$
{with
$\cF(\rho)<+\infty$ (the case $\cF(\rho)=+\infty$ is trivial)}. Defining $c^\eps:=\mass/\urho([-1/\eps,1/\eps])$, and $h^\eps:=c^\eps
\chi_{[-1/\eps,1/\eps]}$, we set $$\urho^\eps:=h^\eps\urho= h^\eps u\Leb1+h^\eps\urho^\perp.$$
Since $\lim_{\eps\downarrow0}h^\eps(x)=1$ pointwise,
for every  function $W:\R\to \R$ such that $\displaystyle \int_\R |W(x)|\,\d\urho(x) <+\infty$, the dominated convergence theorem shows that
\begin{equation}\label{contW}
    \lim_{\eps\downarrow0} \int_\R W(x) \d \urho^\eps(x) =
    \lim_{\eps\downarrow0}\Big(\int_\R W(x) h^\eps(x)u(x) \,\d x + \int_\R W(x) h^\eps(x)\,\d\urho^\perp(x)\Big)
    = \int_\R W(x) \,\d\urho(x).
\end{equation}
In particular, choosing {$W=\varphi$ as in \eqref{eq:14} we obtain that
$W_2(\urho^\eps,\urho)\to 0$ so that for every $\delta>0$ there exists
$R>0$ such that
\begin{equation}
  \label{eq:89}
  \lim_{\eps\downarrow0}\int_\R |x|^2(1-\nchi(x/R))\,\d\rho^\eps=
  \int_\R |x|^2(1-\nchi(x/R))\,\d\rho\le \delta
\end{equation}
for every function $\nchi$ as in \eqref{eq:88}.
On the other hand, \eqref{eq:85} yields $\eps_0>0$ such that
$V^\eps(x)\le V(x)+\delta$ if $|x|\le 2R$ and therefore
\begin{align}
  \notag\cV^\eps(\rho^\eps)&\le
  \int_\R V^\eps(x)\nchi(x/R)\,\d\rho^\eps+
  \int_\R \Big(V(x)+A|x|^2\Big)(1-\nchi(x/R))\,\d\rho^\eps\\
  \label{eq:91}&\le
  \int_\R V(x)\,\d\rho^\eps+\delta\mass+A\int_\R |x|^2(1-\nchi(x/R))\,\d\rho^\eps.
\end{align}
Using \eqref{contW} with $W=V$, passing to the limit as $\eps\downarrow0$ in
\eqref{eq:91}, we obtain
\begin{equation}
  \label{eq:92}
  \limsup_{\eps\downarrow0}\cV^\eps(\rho^\eps)\le V(\rho)+\delta(\mass+A).
\end{equation}
Since $\delta>0$ is arbitrary we conclude
}
% Let us assume that $\int_\R V(x)\d\urho(x)<+\infty$. Since \eqref{crescita-quadratica} shows that the negative part of $V$ grows at most
% quadratically, i.e. $(V)^-(x)\leq C(1+x^2)$, choosing $W=V$ in \eqref{contW} we obtain
\begin{equation}\label{contV}
  \limsup_{\eps\downarrow0} \cV^\eps(\urho^\eps)\le \cV(\rho).
\end{equation}
On the other hand, since $E\leq 0$ is continuous, by Fatou's Lemma we have
\begin{equation}\label{limsupinE}
    \limsup_{\eps\downarrow0}\int_\R E(u^\eps(x))dx \le  \int_\R E(u(x))dx.
\end{equation}
Denoting by $\PlusMeasuresComp{\R,\mass}$ the set of nonnegative measures with compact support and total mass $\mass$, we have just proved that
\begin{equation}\label{Gammacont}
    \forall\,\urho\in\PlusMeasuresTwo{\R,\mass}\cap D(\cF)\ \ \exists \{\urho^\eps\}\subset \PlusMeasuresComp{\R,\mass}
    :\quad
    W_2(\urho^\eps,\urho)\to 0,\quad \lim_{\eps\downarrow0}\cF(\urho^\eps)=\cF(\urho).
\end{equation}
%\GGr{Moreover \eqref{Gammacont} holds also in the trivial case of $\int_\R V(x)\d\urho(x)=+\infty$.}
{A standard diagonal argument for $\Gamma$-convergence shows that
(ii) can be reduced to prove }
\begin{equation}\label{GammaLimsupComp}
    \forall\,\urho\in\PlusMeasuresComp{\R,\mass}, \quad \exists \{\urho^\eps\}\subset \PlusMeasuresComp{\R,\mass}
    : W_2(\urho^\eps,\urho)\to 0,\quad \limsup_{\eps\downarrow0}\cF^\eps(\urho^\eps)\leq\cF(\urho).
\end{equation}
Let $\urho=u\Leb1+\urho^\perp\in\PlusMeasuresComp{\R,\mass}$; denoting
by $k^\eps=\eps^{-1}k(\cdot/\eps)$
a standard family
of {symmetric and nonnegative} mollifiers with support
$[-\eps,\eps]$, we set $u^\eps(x) = (k^\eps * \urho) (x) = \int_\R k^\eps(x-y)\,\d\urho(y)$ and $\urho^\eps=u^\eps\Leb{1}$. By definition of
convolution and Fubini's theorem we have
$$\int_\R V(x)\,\d\urho^\eps(x)= \int_\R V(x)\int_\R k^\eps(x-y)\,\d\urho(y)\,\dx =
\int_{\supp (\urho)}\int_{[-1,1]}V(y+\eps z)k(z)\,\d z\, \d\urho(y).$$
By the continuity of $V$,
%{the fact that $V+\frac{\lambda^-}2|\cdot|^2$ is convex, Jensen inequality,}
and the dominated convergence theorem
\begin{equation}\label{contV2}
%  {\int_{-1}^1V(y+\eps z)k(z)\,\d z\le
 % V(y)+\frac{\lambda^-}2y^2,\quad
  \lim_{\eps\downarrow0} \int_\R V(x) \d \urho^\eps(x) = \int_\R V(x) \d \urho(x).
\end{equation}
Recalling that $E$ is decreasing and applying Jensen's inequality to
the probability measure $k^\eps (x-y)\Leb{1}(y)$ and the convex function $E$ we get
$$ E(u^\eps(x))=E\Big(\int_\R k^\eps(x-y)\,\d \urho(y)\Big)\le E\Big(\int_\R u(y)k^\eps(x-y)\,\d y\Big) \leq \int_\R E(u(y)) k^\eps(x-y)\,\d y.$$
Integrating with respect to $x$ and using Fubini's theorem we obtain
\begin{equation}\label{supE}
    \int_\R E(u^\eps(x))\,\d x \leq \int_\R E(u(x))\,\d x.
\end{equation}
Finally, since $k^\eps\le 1/\eps$ and $u^\eps(x)\le \eps^{-1}\mass$, we have
\begin{equation}\label{estEntr}
%-C\,\eps  \le
  \eps\int_\R u^\eps\log u^\eps \,\d x\le \mass\, \eps \log \frac \mass  \eps.
\end{equation}
%where the lower estimate follows from \eqref{stima-ulogu}.
Since $W_2(\urho^\eps,\urho)\to 0$,
%\GGr{the same argument used for
%  the proof of \eqref{contV2}, together with}
\eqref{contV2}, \eqref{supE} and
\eqref{estEntr} yield \eqref{GammaLimsupComp}.
%\GGr{The proof of (ii) now follows from \eqref{Gammacont} and \eqref{GammaLimsupComp}.}
\end{proof}

\section{Proofs of the main Theorems}

%%\section{Wasserstein gradient flow}\label{sec:gf}

\subsection{Subdifferential characterization of the gradient flow of $\cF$ and existence result.}\label{sec:gf}

\begin{proof}[Proof of Theorem \ref{thm:main1}.]
The proof of Theorem \ref{thm:main1} is based on the general results
about the generation of gradient flows for
displacement $\lambda$-convex functionals in
$\PlusMeasuresTwo{\R,\mass}$ established in \cite{ags} (notice that all the theory in \cite{ags} can be applied to the space
$\PlusMeasuresTwo{\R,\mass}$ and not only to the space $\PlusMeasuresTwo{\R,1}$ considered in \cite{ags}).

By Proposition \ref{prop:lscF} the functional $\cF$ is displacement
$\lambda$-convex
  (in dimension $1$ generalized geodesics \cite[Definition 9.2.2]{ags}
  coincide
  with the displacement interpolations \eqref{eq:38})
% also along generalized geodesics considered in \cite{ags}, see
% Proposition 9.3.2 and 9.3.9)
and we can apply the general theory summarized in Theorem 11.2.1
\cite{ags}.

  Since $D(\cF)=\{\urho\in \PlusMeasuresTwo{\R,\mass}:
  \cF(\urho)<+\infty \}$ is dense in $\PlusMeasuresTwo{\R,\mass}$,
%\GGr{Notice that the closure with respect to
% the $\PlusMeasuresTwo{\R,\mass}$
% convergence of $D(\cF)=\{\urho\in \PlusMeasuresTwo{\R,\mass}: \cF(\urho)<+\infty \}$ coincides with all
% $\PlusMeasuresTwo{\R,\mass}$ (see \eqref{Gammacont} taking into account that $\PlusMeasuresComp{\R,\mass}$ is contained in $D(\cF)$). By the
% previous remark we can consider as initial datum}
  the evolution is well defined starting from
an arbitrary element of $\PlusMeasuresTwo{\R,\mass}$. Therefore, by \cite[Theorem 11.2.1]{ags},
for every $\urho_0\in \PlusMeasuresTwo{\R,\mass}$ there exists a unique curve $\urho$ belonging to  $
C^0([0,+\infty);\PlusMeasuresTwo{\R,\mass})$ such that
$\rho_t\in D(\cI)\subset D(\cF)$ for every $t>0$ and
\begin{align}\label{ce}
  \de_t\urho_t + \de_x(\urho_t\,\vv_t)&=0, \qquad &\text{ in } \DD'(\R\times (0,+\infty)),\\
\label{nlrel}
  \vv_t&=-\de^\circ\cF(\urho_t), \qquad &\text{ for $\Leb{1}$-a.e. } t \in (0,+\infty),\\
%\end{equation}
%\begin{equation}
  \label{eq:43}
  \cF(\urho_{t_0})-\cF(\urho_{t_1}) &= \int_{t_0}^{t_1}\int_\R|\vv_t|^2\,\urho_t(x)\,\d t
  \qquad &0\leq t_0 < t_1,
 %= \int_{t_0}^{t_1}|\partial\cF|^2(\urho_t)\,\d t<+\infty.
\end{align}
Moreover the map $\urho_0\mapsto S_t(\urho_0):=\urho_t$ defines a continuous semigroup satisfying the $\lambda$-contraction property
\eqref{eq:15}.
From \cite[Theorem 2.4.15]{ags} the map $t\mapsto e^{\lambda t}|\de \cF|^2(\urho_t)$ is non-increasing, and then \eqref{eq:19}
holds. The regularization estimate \eqref{eq:18}
  (which implies \eqref{eq:3bis})
  still follows by
  Theorem 11.2.1 and by
\cite{S10} in the case $\lambda\neq 0$.
From \eqref{nlrel} and Theorem \ref{th:charsubdiff} we have
\eqref{eq:3}.
%\GGr{Since by the regularization estimate \eqref{eq:18} we have that $\cF(\urho_t)<+\infty$ for every $t>0$, then by \eqref{eq:43},
%\eqref{nlrel} and \eqref{eq:44} we obtain \eqref{eq:3bis}.}
\eqref{ce}, \eqref{nlrel}, and \eqref{eq:50} yields
\eqref{eq-debole}.
{The comparison result follows from Theorem
\ref{thm:main2} and the corresponding property for solution of the
viscous regularization.}
\end{proof}

\begin{proof}[Proof of Theorem \ref{thm:main2}]
The part concerning existence of solutions to problem \eqref{eq:17} for a measure initial datum, is similar to the part concerning existence for
problem \eqref{eq:rho},  taking into account the characterization of the subdifferential of $\cF^\eps$ \eqref{eq:48}.

The stability with respect to the convergence in $\PlusMeasuresTwo{\R,\mass}$ follows from Lemma \ref{le:Gamma-convergence} and Theorem 11.2.1 of
\cite{ags}. The uniform convergence follows from Theorem \ref{thm:lsc-dissipation} \eqref{eq:219}.
\end{proof}

%
%
%\begin{theorem}\label{th:existenceGF}
%For every $\urho_0\in \PlusMeasuresTwo{\R,\mass}$
%there exists a unique curve $\urho\in C^0([0,+\infty);\PlusMeasuresTwo{\R,\mass})$ such that
%\begin{equation}
%  \de_t\urho_t + \de_x(\urho_t\,\vv_t)=0, \qquad \text{ in } \DD'(\R\times (0,+\infty)),
%\end{equation}
%\begin{equation}
%  \vv_t=-\de^\circ\cF(\urho_t), \qquad \text{ for $\Leb{1}$-a.e. } t \in (0,+\infty),
%\end{equation}
%\begin{equation}
%  \label{eq:43}
%  \int_{t_0}^{t_1}\int_\R|\vv_t|^2\,\urho_t(x)\,\d t=
%  \int_{t_0}^{t_1}|\partial\cF|^2(\urho_t)\,\d t<+\infty.
%\end{equation}
%\end{theorem}
%%(Riportare l'enunciato del Teorema 11.2.1 di \cite{ags})

\subsection{Localized entropy estimates and propagation of
  singularities}\label{sec:estimates}
Let us consider
  \begin{subequations}
    \begin{gather}
  \label{eq:93}
  \text{a smooth convex function $\psi:[0,+\infty)\to \R$ with
    $\psi(0)=0$,}
  \intertext{and let us set (recall that $\beta^\eps(r)=\beta(r)+\eps r$)}
  \label{eq:63}
  \eta(r):=r\psi'(r)-\psi(r),\quad
  \gamma(r):=\int_0^r\beta'(s)\psi'(s)\,\d s,\quad
  \gamma^\eps(r):=  \gamma(r)+\eps\psi(r)=
  \int_0^r
  (\beta^\eps)'(s)\psi'(s)\,\d s.
\end{gather}
\end{subequations}

\begin{theorem}          \label{thm:propagsing}
  If $u^\eps$ is a smooth {bounded} solution to \eqref{eq:17}
  {and $\psi,\eta,\gamma^\eps$ satisfy \eqref{eq:93} and
    \eqref{eq:63},}
  then
  $\psi(u^\eps)$ is a classical solution to
  \begin{equation}  \label{eq:59bis}
        \partial_t\psi(u^\eps)-\partial_x\big(\partial_x\gamma^\eps(u^\eps)+\psi(u^\eps)(V^\eps)'\big)\le \eta(u^\eps)(V^\eps)''.
  \end{equation}
  In particular, for every nonnegative $\phi\in C^2_{\rm c}(\R\times [0,T])$
  it holds
  \begin{equation}
  \label{eq:59}
  \begin{aligned}
  \int_\R \psi(u^\eps(x,T))& \phi(x,T)\,\d x +
    \int_0^T\int_\R
    \psi(u^\eps)\big(-\partial_t\phi+\partial_x \phi (V^\eps)'\big)\,\d x\,\d t\\&
   - \int_0^T\int_\R \big(\gamma^\eps(u^\eps)\partial^2_x\phi +\eta(u^\eps)\phi (V^\eps)''\big)\,\d x\,\d t
     \le \int_\R \psi(u^\eps(x,0))\phi(x,0)\,\d x.
    \end{aligned}
 \end{equation}
\end{theorem}

\begin{proof}
By straightforward computations we obtain that
$$
 \partial_t\psi(u^\eps)-\partial_x\big(\partial_x\gamma^\eps(u^\eps)+\psi(u^\eps)(V^\eps)'\big)
 =\eta(u^\eps)(V^\eps)''-(\beta^\eps)'(u^\eps)\psi''(u^\eps)(\partial_x u^\eps)^2.
$$
Since $\psi$ is convex and $\beta^\eps$ is increasing,  $(\beta^\eps)'(u^\eps)\psi''(u^\eps)(\partial_x u^\eps)^2\ge 0$. This implies
\eqref{eq:59bis}.
\end{proof}

\noindent We will now prove the {\em a priori} estimate \eqref{eq:9}.
\begin{corollary}
  \label{cor:linfty_bound}
  Let us assume that \eqref{eq:8} holds and that $\urho_0=u_0\Leb1$ has a bounded density. Then \eqref{eq:9} holds.
\end{corollary}
\begin{proof}
  {By Theorem
  \ref{thm:main2} it is sufficient to show  \eqref{eq:9} for the
  (bounded and integrable)} solutions $\rho^\eps=u^\eps\Leb{1}$ of \eqref{eq:17} with initial
  datum $\urho_0$.
%r{Assume that $u^\eps$ is bounded.}
Let us apply \eqref{eq:59} with $\psi(r)=r^p$,  $p\ge 2$, and
$\phi(x)=\nchi(x/n)$, where $\nchi$ satisfies \eqref{eq:88}.
{Since $(V^\eps)'$ is bounded and $(V^\eps)''\le \sfc$,}
it is not difficult to pass to the limit as $n \to+\infty$, getting
\begin{equation*}
  \int_\R u^\eps(x,T)^p \,\d x
     \le \int_\R u_0^p(x)\,\d x+\sfc (p-1) \int_0^T\int_\R u^\eps(x,t)^p \,\d x\,\d t.
    \end{equation*}
From Gronwall's Lemma it follows that
\begin{equation*}
\int_\R u^\eps(x,T)^p \,\d x \le \rme^{\sfc (p-1) T}\int_\R u_0^p(x)\,\d x, \qquad \text{for all}\ T>0.
\end{equation*}
Letting $p\uparrow +\infty$ we get estimate \eqref{eq:9} for $\urho^\eps$.
 \end{proof}

 \noindent
The following corollary of Theorem  \ref{thm:propagsing}
is a preliminary step for the proof of Theorem \ref{thm:main3}
on the propagation of the singularities.
\begin{corollary}\label{cor3}
  Let $\psi,\eta,\gamma$ be as in \eqref{eq:93} and \eqref{eq:63},
    with
    $\lim_{r\uparrow+\infty}\psi'(r)=\psi_\infty'\in (0,+\infty)$.
    % and $\lim_{r\uparrow+\infty}\psi''(r)=0$.
    If $\urho=u\Leb 1+\urho^\perp$ is the measure-valued solution to
  \eqref{eq:rho}
  and $\psi(\urho):=\psi(u)\Leb 1+\psi_\infty'\,\urho^\perp$, we have
  \begin{equation}
    \label{eq:66}
    \partial_t
    \psi(\urho)-\partial_x\big(\psi(\urho)V'\big)\le \partial^2_x(\gamma(u))+\eta(u)V''\quad
    \text{in the sense of distributions}.
  \end{equation}
\end{corollary}
\begin{proof}
  It is sufficient to pass to the limit in \eqref{eq:59}, recalling \eqref{eq:65}
   and applying the dominated convergence theorem with the estimate
   $|\psi(r)|\le \|\psi'\|_{L^\infty((0,+\infty))} r$.

     Notice that
     \begin{displaymath}
       \lim_{r\to+\infty}\frac {\eta(r)}r=
       \lim_{r\to+\infty} \Big(\psi'(r)-\frac {\psi(r)}r\Big)=0
     \end{displaymath}
     and
     \begin{displaymath}
       \lim_{r\to+\infty}\frac {\gamma(r)}r=
       \lim_{r\to+\infty}\frac
       1r\Big(\beta(r)\psi'(r)-\beta(0)\psi'(0)-\int_0^r
       \beta(s)\psi''(s)\,\d s\Big)
       =0,
     \end{displaymath}
     since $\lim_{r\uparrow+\infty}\beta(r)=\beta_\infty<+\infty$ and we estimate the integral as follows
     \begin{displaymath}
     0\le \frac {1}{r} \int_0^r
       \beta(s)\psi''(s)\,\d s\le
      \frac{\beta_\infty}{r}\Big(\psi'(r)-\psi'(0)\Big).
     \end{displaymath}
  \end{proof}

\begin{proof}[Proof of Theorem \ref{thm:main3}]
Let us fix a nonnegative function $\zeta\in C^\infty_{\rm c}(\R)$
  with compact support in $[0,1]$ and integral equal to $1$. We set
  $\zeta_k(r):=\zeta(r-k)$, $Z_k(r):=\int_0^r \zeta_k(s)\,\d s$,
  $\psi_k(r)=\int_0^r Z_k(s)\,\d s$. It is immediate to check that $\psi_k$
  satisfies the assumptions of Corollary \ref{cor3}.  Moreover, the
  corresponding functions $\gamma_k(r)$ and $\eta_k(r)$ are uniformly
  bounded by $Cr$ and converge to $0$ pointwise as $k\to +\infty$.
  Passing to the limit in
  \begin{equation}
    \label{eq:66bis}
    \partial_t
    \psi_k(\urho)-\partial_x\big(\psi_k(\urho)V'\big)\le \partial^2_x(\gamma_k(u))+\eta_k(u)V''\quad
    \text{in the sense of distributions}
  \end{equation}
as $k\uparrow+\infty$ we obtain \eqref{eq:68}.  Now, set $\mu_t=
(\sfX_t)_\#\urho_0^\perp$. It is well known that $\mu_t$ solves
$\partial_t\mu_t-\partial_x(\mu_tV')=0$. Then the family of measures
$\sigma_t=\urho_t^\perp-\mu_t$ satisfies
$\partial_t\sigma_t-\partial_x(\sigma_tV')\le 0$ with {$\sigma_0\le0$}.
By {a simple variant of} Proposition 8.1.7 of \cite{ags} we deduce that $\sigma_t\le 0$
for every $t\ge 0$.
%\GGr{We point out that Proposition 8.1.7 of
%\cite{ags} applies if $\partial_t\sigma_t-\partial_x(\sigma_tV')=
%0$. On the other hand, it can easily verified that the same proof
%continues to hold even in the case
%$\partial_t\sigma_t-\partial_x(\sigma_tV')\le 0$.}
Therefore for
every Borel set $A\subset \R$, $ \urho_t^\perp(A)\le \urho_0^\perp
\big(\sfX_t^{-1}(A)\big)$. Choosing $A=D_t$, the inclusion
$\CDom{u_t}\subset \mathsf J_t$ follows.
\end{proof}

%\begin{corollary}
%  If  $\urho=u\Leb 1+\urho^\perp$ is the measure-valued solution to
%  \eqref{eq:rho} we have
%  \begin{equation}
%    \label{eq:67}
%    \partial_t\urho^\perp-\partial_x\big(\urho^\perp V'\big)\le
%    0\quad\text{in }\mathscr D'(\R\times(0,+\infty))
%  \end{equation}
%\end{corollary}
%\begin{proof}
%  \end{proof}

\subsection{Minimizers, stationary solutions, and asymptotic properties}
\label{sec:min}
\begin{proof}[Proof of Theorem \ref{thm:main4}.]
  Let us first show that
  every measure
  $\urho_{\rm min}=u_{\rm min}\Leb1 +\urho^\perp_{\rm min}$ satisfying
  \eqref{eq:30} is a minimizer for $\cF$.

    Notice that by construction $\rho_{\rm min}\in \CPlusMeasures{\R,\mass}$.

%Set $y=M^{-1}(\min(\mass,\mass_\rmc))$ so that $u_{\rm min}(x)=H(V(x)-y)$ in both cases.
Let {$\urho=u\Leb 1+\urho^\perp$ be an arbitrary measure in
$\PlusMeasures{\R,\mass}$.}
If  $A=\{x\in\R:V(x)-\frak v<\dV \}$
and $B=\R\setminus A$ denotes its complement,
$$ u_{\rm min}(x)=\begin{cases}H(V(x)-\frak v) & \text{if }x\in A,
  \\ 0 & \text{if }x\in B. \end{cases}$$
Since $$ E'(H(v))=\begin{cases}-v & \text{if }v\in(0,\dV), \\
  -\dV
& \text{if }v\in[\dV,+\infty), \end{cases}$$
and $E$ is convex, we
get
\begin{align*}
  \cE(\urho)-\cE(\urho_{\rm min})& =
  \int_\R \big(E(u(x))-E(u_{\rm min}(x))\big)\,\d x\ge
  \int_\R E'(u_{\rm min}(x))\big(u(x)-u_{\rm min}(x)\big)\,\d x\\&=
  \int_A (\frak v-V(x)) \big(u(x)-u_{\rm min}(x)\big)\,\d x
  - \dV\int_B  u(x)\,\d x.
\end{align*}
Moreover, since $V(x)-\frak v\ge \dV$
for every $x\in B$,
\begin{align*}
  \cF(\urho)-\cF(\urho_{\rm min})&=\cE(\urho)-\cE(\urho_{\rm min})+\int_\R
  V\,\d\urho-\int_\R V\,\d\urho_{\rm min}\\
  &\ge \int_A (\frak v-V(x)) \big(u(x)-u_{\rm min}(x)\big)\,\d x
  + \int_B (V(x)-\dV) u(x)\,\d x \\
  &\quad + \int_A V(x) \big(u(x)-u_{\rm min}(x)\big)\,\d x
  +\int_\R V\,\d\urho^\perp-\int_\R V\,\d\urho^\perp_{\rm min} \\
  &\ge \int_\R \frak v \big(u(x)-u_{\rm min}(x)\big)\,\d x
  +\int_\R V\,\d\urho^\perp-\int_\R V\,\d\urho^\perp_{\rm min}.
\end{align*}
 Hence, owing to the identity
 \[\displaystyle
 {\rho(\R)=\rho_{\rm min}(\R),\quad
 \text{so that}\quad
 \int_\R u\,\d x-
 \int_\R u_{\rm
min}\,\d x=\int_\R \d \urho_{\rm min}^\perp-\int_\R \d \urho^\perp } ,
 \]
  and  recalling that $\urho_{\rm
min}^\perp$ is concentrated in $Q$, we obtain
\begin{align*}
  \cF(\urho)-\cF(\urho_{\rm min})&\ge\int_\R \frak v \big(u(x)-u_{\rm min}(x)\big)\,\d x
  +\int_\R V\,\d\urho^\perp-\int_\R V\,\d\urho^\perp_{\rm min}\\
  &=\int_\R (V-\frak v)\,\d\urho^\perp-\int_\R(V-\frak v)\,\d\urho^\perp_{\rm min}
  \ge -\int_\R(V-\frak v)\,\d\urho^\perp_{\rm min} =0.
\end{align*}
This shows that  $\cF(\urho)\ge\cF(\urho_{\rm min})$ for every
{$\urho\in\PlusMeasures{\R,\mass}$.}

  We prove now that every minimizer $\urho=u\Leb 1+\urho^\perp \in
  \PlusMeasures{\R,\mass}$ of $\cF$ in $\PlusMeasures{\R,\mass}$
  satisfies \eqref{eq:30}. We
  consider another minimizer $\rho_{\rm min}$ given by \eqref{eq:30}
  so that
equalities hold in all the previous inequalities and in particular we have
$$
0=\cF(\urho)-\cF(\urho_{\rm min})= \int_\R (V-\frak v)\,\d\urho^\perp .
$$
It follows that $\urho^\perp$ is concentrated on $Q$ and $\urho^\perp =0$ when $\mass < \mass_{\rm c}$
(recall that $V(x)-\frak v \ge 0$ and equality holds if and only if $\frak v = V_{\rm min}$ and $x\in Q$).
If $u\not = u_{\rm min}$, then, by the strict convexity of $E$,
$\cF((1-\theta)\urho + \theta \urho_{\rm min})< \cF(\urho_{\rm min})$ for every $\theta \in (0,1)$.
Taking the continuity of $u$ into account, it follows that $u(x)=u_{\rm min}(x)$ for every $x\in \R$.
Consequently
$\urho^\perp (\R) = \urho_{\rm min}^\perp (\R)$ and we conclude.
\end{proof}

\begin{proof}[Proof of Theorem \ref{thm:Fisher=0}]
  It follows easily by \cite[Theorem 11.1.3]{ags}, which shows
  in particular that $\rho$ is a stationary solution of the Wasserstein gradient
  flow of a displacement $\lambda$-convex functional $\cF$ iff
  $|\partial \cF|(\rho)=0$.
  We can then invoke Theorem \ref{th:charsubdiff}.
\end{proof}
The proof of Theorems \ref{thm:statI} and \ref{thm:main_stationary} is based on the following lemma:
\begin{lemma}
  \label{le:char}
  Let $\rho=u\Leb 1+\rho^\perp\in \CPlusMeasures\R$ be a measure
  satisfying $\cI(\rho)=0$, and let us consider the open set $\Pos u:=\big\{x\in
  \R:u(x)>0\big\}$.
  If $I$ is a connected component of $\Pos u$ then
  \begin{equation}
  \label{eq:64}
  E'(u(x))+V(x)=c_I\quad \text{for every }x\in I.
\end{equation}
\end{lemma}
\begin{proof}
  Let us first show that
  the function
  $E'\circ u$ belongs to $W^{1,1}_{\rm loc}(\Pos u)$ with
  \begin{equation}
    \label{eq:100}
    \partial_x \big(E'\circ u\big)=\frac {\partial_x \big(\beta\circ
      u\big)}u\quad \text{in }\Pos u.
  \end{equation}
  We can simply write $E'\circ u=L\circ (\beta\circ u)$ where
  $L:=E'\circ \beta^{-1}$ and $\beta\circ u\in W^{1,1}_{\rm loc}(\R)$.
  The function $L$ belongs
  to $C^1(0,\betainfty)$ and can be extended to $\betainfty$ by continuity
  setting $L(\betainfty)=0$; it is easy to check that this
  extension belongs to $C^1(0,\betainfty]$, since
  \begin{displaymath}
    L'(r)=\frac{E''\circ \beta^{-1}}{\beta'\circ \beta^{-1}}=
    \frac{1}{\beta^{-1}},
    \quad
    \lim_{r\uparrow\betainfty}L'(r)=0.
  \end{displaymath}
  \eqref{eq:100} then follows by the chain rule for the composition of
  a $C^1$ with a Sobolev function.

%  By Theorem \ref{thm:Fisher=0}
%  $\cI(\rho)=0$;
  If $I$
  is a connected component of $\Pos u$, we have
  \begin{equation}
    \label{eq:70}
    0 = \frac{\partial_x\beta(u(x))}{u(x)}+V'(x) = \partial_x(E'(u(x))+V(x)) \quad\text{in }I,
  \end{equation}
  so that there exists a constant $c_I$ such that \eqref{eq:64} holds.
\end{proof}

\begin{proof}[Proof of Theorem \ref{thm:statI}]
  We have to prove only the ``right'' implication $\Rightarrow$.

  A simple argument by contradictions shows that $\Pos u=\R$:
  otherwise, if the interval
  $I=(a,b)$ is a connected component of $\Pos u$ and one of its
  extremes, say $a$, is finite,
  we should have
  \begin{displaymath}
    \lim_{x\downarrow a}u(x)=0,\quad
    -\dV=\lim_{x\downarrow a}E'(u(x))=c_I-V(a)>-\infty.
  \end{displaymath}
  Since $\Pos u=\R$ Lemma \ref{le:char} yields $V(x)\ge c_I$ for every
  $x\in \R$ and $u(x)=H(V(x)-c_I)$. Since $\rho\in
  \CPlusMeasures{\R,\mass}$
  we conclude that \eqref{eq:30} holds and $\rho$ is a minimizer of
  $\cF$
  by Theorem \ref{thm:main4}.
\end{proof}

\begin{proof}[Proof of Theorem \ref{thm:main_stationary}]
  Let $\rho=u\Leb 1+\rho^\perp\in \CPlusMeasures{\R,\mass}$ with
  $\cI(\rho)=0$ and let $I=(a,b)$ be
  a connected component of the open set $\Pos u$.
  Since the range of the function $r\mapsto -E'(r)$ for $r\in
  (0,+\infty]$ is the bounded interval $(0,\dV]$ and
  $\lim_{|x|\to\infty}V(x)=+\infty$
  we deduce from Lemma \ref{le:char} that $I$ is bounded.

  It follows that $u(a)=u(b)=0$ and therefore
  $\lim_{x\downarrow a}E'(u(x))=
  \lim_{x\uparrow b}E'(u(x))=-\dV$, $c_I=V(a)-\dV=V(b)-\dV$.
  We thus obtain \eqref{eq:98} and the representation \eqref{eq:103},
  which also yields \eqref{eq:99} since $u$ is integrable in $\R$.
  Since for every $x\in I$ $u(x)=+\infty$ iff $V(x)=V(a)-\dV$, i.e.
  $x\in Q_I$, we obtain \eqref{eq:84}.

  Conversely, if $\rho=u\Leb 1+\rho^\perp\in
  \CPlusMeasures{\R,\mass}$ satisfies the
  three conditions of Theorem \ref{thm:main_stationary},
  we immediately have that $\cI(\rho)=0$.
  In fact, the first integral of the definition of $\cI$ in
  \eqref{eq:5}
  vanishes by \eqref{eq:64} and \eqref{eq:100}; the second
  integral, corresponding to the singular part of $\rho$ vanishes
  since $\rho^\perp$ is concentrated on $Q(u)$ and
  $V'$ vanishes in each point of $Q_I$, which is a local minimizer of $V$.
\end{proof}

\begin{proof}[Proof of Corollary \ref{cor:obvious}]
  Remark \ref{rem:examples} shows that the minimizer of $\cF$ is
  unique.
  We have just to check the case when $\dV<+\infty$.
  By the assumption on the first derivative of $V$ is immediate to
  check that the set $\Pos u$ contains just one connected component $I=(a,b)$
  with  $a<q_-<q_+<b$. Theorem \ref{thm:main4}
  shows that $\rho$ is a minimizer of $\cF$.
\end{proof}

\begin{proof}[Proof of Theorem \ref{thm:main5}.]
We use the dissipation identity \eqref{eq:13} to obtain the
inequality
\begin{displaymath}
  {\int_{t_0}^{t_1}\cI(\urho_t)\,\d t=
  \cF(\urho_{t_0})-\cF(\urho_{t_1})\le \cF(\urho_{t_0})-\cF(\bar
  \rho)<+\infty\quad
  \text{for every }0<t_0<t_1<+\infty.}
\end{displaymath}

  Passing to the limit as $t_1\uparrow+\infty$ we get $\cI(\rho_t)\in
L^1(t_0,+\infty)$, so that
\begin{equation}
  \label{eq:106}
  \sum_{n=2}^{+\infty}\int_{n-1}^{n}\cI(\rho_t)\,\d t<+\infty.
\end{equation}
Since by \eqref{eq:19} $\cI(\rho_t)\ge \rme^{-2\lambda^-}\cI(\rho_n)$
if $t\in (n-1,n)$ we obtain $\sum_{n=2}^{+\infty}\cI(\rho_n)<+\infty$;
in particular
\begin{equation}
  \label{eq:108}
  \lim_{n\uparrow+\infty}\cI(\rho_n)=0\quad\text{and a further
    application of \eqref{eq:19} yields}\quad
  \lim_{t\uparrow+\infty}\cI(\rho_t)=0.
\end{equation}
Since
$\cF(\urho_{t})\leq \cF(\urho_{t_0})$ for every $t\geq t_0$, by \eqref{eq:95}
we infer  that {$\{\urho_{t}\}_{t\ge t_0}$ is tight;}
by Theorem \ref{thm:lsc-dissipation2}
any weak limit point $\rho_\infty$ of $\rho_t$ as $t\uparrow+\infty$
satisfies
$\cI(\urho_\infty)=0$ and therefore $\rho_\infty=\bar\rho$. It follows
that $\rho_t\weakto \bar\rho$ weakly as $t\uparrow+\infty$.

Theorem \ref{thm:lsc-dissipation2} yields the uniform convergence of
$u_t$ to $\bar u$
on compact sets of $\Dom{\bar u}$ as $t\to +\infty$.
When $\mass<\mass_\rmc$, $\bar\urho$ has a
bounded density and therefore for every compact subset $K\subset \R$
there exists a time $T>0$ such that $\urho_{t}$ is bounded on $K$
for every $t\ge T$. Choosing as $K:=\big\{x\in \R:V(x)\le c\big\}$
for a constant $c$ sufficiently big so that $K$ contains the support
of $\rho_0^\perp$, Theorem \ref{thm:main3}
shows that the support of $\rho_t^\perp$ is contained in $K$
for every $t>0$ and therefore $\rho_t^\perp=0$ for $t\ge T$.
\end{proof}

\bibliographystyle{siam}
\def\cprime{$'$}
\bibliography{bibliografiaFLST}

\end{document}